\documentclass[10.5pt,reqno,a4paper]{article}
\usepackage{amsmath,amsfonts,amsthm,amssymb,graphicx}
\usepackage[titletoc,title]{appendix}
\usepackage{subfigure}
\usepackage{pdfpages}
\usepackage{alltt}
\usepackage{graphicx}
\usepackage{float}
\usepackage{enumitem}
\usepackage{pgfplots}
\usepackage[final]{showkeys} 
\usepackage{cleveref}
\usepackage{ mathrsfs }
\newtheorem{theorem}{Theorem}[section]
\newtheorem{lemma}[theorem]{Lemma}
\newtheorem{claim}[theorem]{Claim}
\newtheorem{proposition}[theorem]{Proposition}

\newcommand{\subscript}[2]{$#1 _ #2$}
\theoremstyle{definition}
\newtheorem{definition}[theorem]{Definition}

\newcommand{\dist}{\mathrm{dist}}
\newcommand{\e}{\varepsilon}
\newcommand{\bD}{\mathbb{D}}
\newcommand{\bN}{\mathbb{N}}
\newcommand{\bC}{\mathbb{C}}

\newcommand{\bR}{\mathbb{R}}

\newcommand{\bdry}{\partial \bD}
\newcommand{\puncD}{\bD^{*}}
\newcommand{\cale}{\mathcal{E}_\e}
\newcommand{\calk}{\mathcal{K}}
\newcommand{\call}{\mathcal{L}}

\newcommand{\scre}{\mathscr{E}}
\newcommand{\scrj}{\mathscr{J}}

\newcommand{\ma}{\mathbf{a}}

\newcommand{\wma}{\widehat{\mathbf{a}}}
\newcommand{\wmb}{\widehat{\mathbf{b}}}
\newcommand{\wmy}{\widehat{\mathbf{y}}}
\newcommand{\wmp}{\widehat{\mathbf{p}}}
\newcommand{\wmq}{\widehat{\mathbf{q}}}

\newcommand{\calh}{\mathcal{H}}

\newtheorem{remark}[theorem]{Remark}

\numberwithin{equation}{section}
\newcommand{\bZ}{\mathbb{Z}}
\begin{document}
\title{Periodic Orbits of Gross Pitaevskii in the Disc with Vortices Following Point Vortex Flow}
\author{Raghavendra Venkatraman\footnote{email: rvenkatr@umail.iu.edu} \\ \small Indiana University, Bloomington.}

\maketitle

\begin{abstract}
We prove the existence of non-constant time periodic vortex solutions to the Gross-Pitaevskii equations for small but \textit{fixed} $\e > 0.$ The vortices of these solutions follow periodic orbits to the point vortex system of ordinary differential equations \textit{for all time}. The construction uses two approaches-- constrained minimization techniques adapted from \cite{GS} and topological minimax techniques adapted from \cite{LinMinMax}, applied to a formulation of the problem within a rotational ansatz. 
\end{abstract}
\section{Introduction} \label{sec:intro}
We study non-constant time periodic vortex solutions to the Gross-Pitaevskii (GP) equations 
\begin{align}
\label{eq:GP} iu_t (x,t) &= \Delta u (x,t) + \frac{u(x,t)(1-|u(x,t)|^2)}{\e^2}, \hspace{1cm} &(x,t) \in \bD \times \bR^+, 
\end{align}
subject to Dirichlet boundary conditions
\begin{align}
\label{eq:bc}
u(e^{i \theta},t) &= g_n(\theta) := e^{in \theta}, \hspace{6.4cm}  \theta \in [0,2\pi), t \in \bR^+. 
\end{align}
Here $\bD \subset \bC \cong \bR^2$ is the unit disc, $n$ is a non-negative integer and $u:\bD \times \bR \to \bC.$ We are interested in studying solutions to this system for small but fixed $\e > 0.$ 
\par 
The GP equations arise in a number of areas of physics such as superfluids, Bose-Einstein condensation and non-linear optics; see for instance \cite{bosecondensate}. The time-periodic solutions to GP that we construct are rotational, in the sense of an ansatz that we explain below in \eqref{ansatz}. Roughly speaking, the vortices of our time periodic solutions rotate at a uniform angular velocity about the origin of the unit disc. The vortex configurations from the celebrated rotating bucket experiment, cf. \cite{stirvortex}, bear resemblance to those explored in this paper. However the emergence of vortices in experiments is achieved through an externally driven rotation of the domain, whereas here the presence of vortices is topologically enforced through the Dirichlet boundary condition. 
\par
The GP equations are a Hamiltonian system. Consequently, the flow defined by solutions to \eqref{eq:GP} conserves the Ginzburg-Landau energy 
\begin{align} \label{gl}
E_\e(u) := \frac{1}{2}\int_{\bD} |\nabla u|^2 + \frac{1}{2\e^2}(1-|u|^2)^2\,dx. 
\end{align}
The $\e \to 0$ asymptotics of global minimizers $u_\e$ of \eqref{gl} under a Dirichlet boundary condition was the subject of the seminal work \cite{BBH}. This work identifies ``vortices" emerging in the $\e \to 0$ limit whose locations are determined through minimization of a finite dimensional renormalized energy denoted $W.$ The renormalized energy is obtained by subtracting off the logarithmically divergent part of $E_\e(u_\e).$ 
Akin to the Gross-Pitaevskii system being the $L^2-$Hamiltonian flow associated to the Ginzburg Landau energy $E_\e,$ the Hamiltonian flow in $\mathbb{C}^N$ associated to the renormalized energy $W$ is referred to as the point vortex system of ODE's, cf. \eqref{pvfode}. This system of ODE's also arises in fluid mechanics, and corresponds to the incompressible Euler equations with vorticity given by a sum of Dirac deltas and constant tangential velocity boundary conditions; see \cite[Chapter 4]{marchioropulvirenti} for a precise statement. 
\par
Rigorous asymptotic analysis for the GP equations posed on $\bR^2$ and $\mathbb{T}^2$ in the $\e \to 0$ limit and its connection to the point vortex system was initiated in \cite{colljerrard,colljerrard2,lin1999incompressible}. A crucial measure theoretic tool in these and other studies of Ginzburg-Landau is the Jacobian $Ju$ of a function $u,$ which measures concentration of derivatives of the phase of the function $u$, cf. \cite{JerrardSoner}. Starting with a sequence of energetically well-prepared initial data whose Jacobian measures are close to a sum of Dirac masses $(a_i^0)_{i=1}^N$ and denoting the resulting point vortex flow by $(a_i(t))_{t \in [0,T)},$ it is shown in \cite{colljerrard,colljerrard2,lin1999incompressible} that along a subsequence $\{\e_j\} \to 0$, at each time $t \in [0,T)$ the Jacobians of the corresponding solutions to GP converge to a sum of Dirac masses supported on $(a_i(t)).$ Here $T$ is the first time of collision of any two vortices for the point vortex flow. It is independent of $\e,$ and is possibly infinite. 
\par
More recently, the authors in \cite{jerrardspirn} proved refined estimates relating GP dynamics to those of point vortices, focusing on estimates that hold for possibly small but \textit{fixed} $\e > 0$. These estimates yield in particular that in the presence of a bounded number of vortices and given sufficiently small $\e$, the Jacobians of solutions $u_\e(\cdot, t)$ to GP are \textit{quantitatively} in $\e$ close to point vortex dynamics $(a_i(t)),$ the estimates holding on time intervals of order $0 \leq t \lesssim |\ln \e|$. As a corollary, if $(a_i(t))$ represents a periodic orbit to point vortex dynamics, the results of \cite{jerrardspirn} show that for fixed small $\e > 0,$ the Jacobians of $u_\e(\cdot,t)$ are \textit{close} to periodic functions for $0 \leq t \lesssim |\ln \e|.$ 
\par 
In light of the foregoing discussions, it is thus natural to ask for existence of \textit{periodic orbits} to GP for small but \textit{fixed} $\e > 0$ that follow point vortex dynamics \textit{for all time.} Addressing these issues in the symmetric situation of \eqref{eq:GP}-\eqref{eq:bc} is the main contribution of our paper. Our main results are contained in Theorems \ref{gpsoltheorem}, \ref{thm:existence}, and \ref{thm:asymptotics}. For the sake of brevity and to avoid technicalities, we formulate a unified statement of our main results in Theorem \ref{thm:main} below. 
\par 
We phrase our result in terms of relative equilibria to the point vortex flow, namely periodic solutions whose vortices appear stationary from a uniformly rotating frame. They consist of nested rings of vortices within the disc, uniformly rotating about its center; the number of vortices per ring is constant and all vortices in a given ring have the same degree. However, different rings might have vortices of varying degrees. In this paper, we are exclusively interested in relative equilibria with vortices of degree $\pm 1.$ We discuss relative equilibria in the unit disc in Section \ref{sec:PVF}, where we characterize them algebraically. To the best of our knowledge, this general characterization of relative equilibria in the unit disc is new.
\par 
Our main results are the content of Theorems \ref{gpsoltheorem},\ref{thm:existence} and \ref{thm:asymptotics} but we condense them here into one general- and less detailed statement. 
\begin{theorem} \label{thm:main}
Fix a relative equilibrium $(a_i(t), d_i)_{i=1}^N,$ where $a_i(t) \in \bD$ and $d_i \in \{\pm 1\}$ are the associated degrees satisfying $\sum_i d_i = n.$ Then there exists $\e_0 = \e_0(a_i)$ such that for all $0 < \e < \e_0,$ there exist smooth, non-constant time periodic vortex solutions $u_\e(x,t)$ of the Gross-Pitaevskii equations \eqref{eq:GP}-\eqref{eq:bc} and a fixed number $\theta_* \in [0,2\pi)$ such that for each $t \geq 0,$ as $\e \to 0,$ there holds
\begin{align}
Ju_\e(x,t) := \det (\nabla_x u_\e (x,t))  \,dx \rightharpoonup \pi \sum_{j=0}^N d_j \delta_{e^{i \theta_*}a_j(t)}
\end{align}
in the weak topology of $W^{-1,1}(\bD).$ 
\end{theorem}
\par
We next lay out the rotational framework that underlies the method of proof of Theorem \ref{thm:main}. For any positive integer $k,$ we say that a function $u(x) = u(r,\theta)$ enjoys a $k-$fold symmetry property if  
\begin{align} \label{kfoldsym}
u\left(r,\theta + \frac{2\pi}{k}\right) = u(r,\theta), \hspace{1cm} r \in (0,1],\theta \in [0,2\pi].
\end{align}
In order for a function with the $k-$fold symmetry property to be compatible with the boundary condition \eqref{eq:bc}, it is necessary that $k$ divides $n.$ Setting $n=km,$ we pursue the ansatz
 \begin{align} \label{ansatz}
u(x,t) = R(-k\omega t) v\big(R(\tfrac{\omega}{m} t)x\big),
\end{align}
where $\omega \in \bR$ is an unknown angular velocity and $v:\bD \to \bC;$ for any real number $\beta,$ the matrix $R(\beta)$ is the rotation matrix by angle $\beta$ given by 
\begin{align*}
R(\beta) = \begin{pmatrix}
\cos \beta & -\sin \beta \\
\sin \beta & \cos \beta
\end{pmatrix}.
\end{align*}
The function $u$ has the $k-$fold symmetry property if and only if $v$ does. Observe from \eqref{ansatz} that a function $u = u(x,t): \bD \times \bR \to \bC$  which enjoys the $k-$fold symmetry property for each $t \in \bR$ is periodic with time period 
\begin{align*}
T = \frac{2\pi m}{|\omega|}.
\end{align*}
We note that in defining $m$ in the ansatz \eqref{ansatz} above, we have implicitly assumed that $n \geq 1$ to ensure that $m \neq 0.$ The case $n = 0$ can easily be fit into this framework: simply take $k = 0$ and $m$ a positive integer that reflects the desired symmetry property of the function $v.$ Furthermore, if the function $v$ has degree $n,$ and has the $k-$fold symmetry property for a divisor $k$ of $n,$ it is clear that the number of $+1$ and $-1$ vortices are all divisible by $k.$ 
\par
Using ansatz \eqref{ansatz} in \eqref{eq:GP} yields the elliptic system
\begin{align} \label{ansatzpde}
\Delta v(y) + \frac{v}{\e^2}(1-|v|^2)(y) &= \omega\left(k v + \frac{1}{m} y^\perp \cdot \nabla v^\perp\right)(y), \hspace{1cm} y \in \bD,\\ \label{ansatzbc}
v(e^{i \theta}) &= e^{i n \theta}.
\end{align}
\par 
Our strategy in proving Theorem \ref{thm:main} is variational and symmetry plays a crucial role in the development. The rotational invariance of both the Ginzburg Landau energy as well as the renormalized energy yields by Noether's theorem, additional conserved quantities captured by the functional
\begin{align} \label{momconstraint}
\scrj(v) := -\frac{1}{2}\int_\bD k|v|^2 + \frac{1}{m} v \cdot (y^\perp \cdot \nabla) v^\perp \,dy
\end{align}
for the GP system, and $\scrj_0$ defined in \eqref{eq:fdmomentum} for the point vortex system. 
\par 
The elliptic system \eqref{ansatzpde} has two related variational viewpoints, a constrained minimization viewpoint adopted in Section 3, and a topological one pursued in Section 4. 
\par 
The main result of Section \ref{sec:Min}, Theorem \ref{gpsoltheorem} covers the case of relative equilibria with a single vortex ring containing vortices of degree either $1$ or $-1,$ but not both. The proof of Theorem \ref{gpsoltheorem} proceeds by the strategy of constrained minimization followed by an application of the vortex balls construction; the method is adapted from \cite{GS}. The constrained minimization step involves minimizing the Ginzburg Landau energy within a fixed symmetry class, and with the constraint $\scrj$ value equal to the $\scrj_0$ value of the given relative equilibrium. The unknown angular momentum $\omega_\e$ arises as a Lagrange multiplier associated to the momentum constraint. The advantage of this method is that it is simple and relies on the complete integrability of single vortex rings. The disadvantage of this approach however, is that we are so far unable to obtain $\e-$independent bounds on the Lagrange multipliers $\omega_\e.$ This restricts the analysis of even multi-ring solutions having only degree $+1$ vortices.
\par The second approach, based on a linking argument in the spirit of \cite{LinMinMax}, is contained in Section \ref{sec:Lin} and allows us to handle any relative equilibrium, including those with a mix of $+1$ and $-1$ vortices. The main results of Section 4 are contained in Theorems \ref{thm:existence} and \ref{thm:asymptotics}. The linking approach is more involved, and we present a sketch of the proof at the outset of Section \ref{sec:Lin} before describing it in detail. 
\par 
In Appendix A we collect a few technical results that prove useful in Sections 3 and 4. 
\par 
\vspace{.4cm}
\underline{\textbf{Notation:}} We frequently identify a complex number $z = z_1 + i z_2,$ where $z_1,z_2 \in \bR,$ with the point $(z_1,z_2) \in \bR^2.$ Consequently, given a pair of complex numbers $z,w = w_1 + i w_2,$ we write $(z,w) := z_1 w_1 + z_2 w_2,$ i.e. their dot product as vectors in $\bR^2.$ We emphasize that this is a real inner product. By abuse of notation, given $u \in \bC$ and $v = (v_1,v_2) \in \bC^2,$ we also write $(u,v) = ((u,v_1),(u,v_2)) \in \bR^2.$ We denote $J := R(\pi/2).$  Given $v = (v_1,v_2) \in \bR^2,$ we write $v^\perp :=  (-v_2,v_1).$ Next, if $v, w \in \bR^2,$ where $v = (v_1,v_2), w= (w_1, w_2),$ then $v \times w = v_1 w_2 - v_2 w_1.$ The ball of radius $R > 0,$ centered at a point $x$ is denoted $B_R(x).$ Unless otherwise mentioned, balls are open. \par 
Generic constants are denoted by $C,$ and the value of $C$ may vary with each occurrence. Such constants are positive and never depend on $\e.$ Constants with subscripts such as $c_1$ stay the same within the proof in which they occur. 
\par 
For any integer $k \geq 1$ we denote as usual by $\bD^k$ the $k-$fold Cartesian product of $\bD$ with itself, and   
\begin{align*}
({\puncD})^k := \{b = (b_1,\cdots, b_k) \in \bD^k, b_i \neq b_j \mbox{ if } i \neq j \}.
\end{align*}
Given a set $S \subset \bC$ and a complex number $\lambda \neq 0,$ we write 
\begin{align}
\label{eq:multipleofset}
\lambda S := \{\lambda s : s \in S\}. 
\end{align}
Finally, since $\mathbb{C}^k$ is a $k-$dimensional complex vector space, for any point $p \in \mathbb{C}^k,$ with $p = (p_1,\cdots, p_k),$ we write $\lambda p = (\lambda p_1,\cdots, \lambda p_k ) \in \bC^k,$ where $\lambda \in \bC.$  

\section*{Acknowledgements}
This work is part of my PhD thesis at Indiana University. I am extremely grateful to my thesis adviser, Professor Peter Sternberg, for his invaluable advice. Thanks are also due to Professor Robert L. Jerrard who suggested the rotational ansatz for the case of a single vortex, i.e. $n = k=m=1$ in Eq. \eqref{ansatz} above, and to the anonymous referee for their several useful comments that resulted in an improved manuscript. Finally, I thank partial support from the NSF through grants D.M.S. 1101290 and D.M.S.1362879. 

%\section*{Notation}
\section{Renormalized Energy and the Point Vortex Flow}  \label{sec:PVF}
In this section, we first recall the renormalized energy identified in \cite{BBH}. Subsequently, we discuss relative equilibria to the point vortex system of ODE's posed in the unit disc, and characterize them algebraically, following similar arguments in \cite{lewisratiu} for the point vortex system on $\bC.$ 
\par The renormalized energy is a function of vortex locations and associated vortex strengths. We recall that for $\mathbf{b} = (b_1, \cdots, b_N) \in ({\puncD})^N,$ and vortex strength $d = (d_1,\cdots, d_N) \in \bZ^N$ with $\sum d_i = n,$
\begin{align}
\label{renorm}
W(\mathbf{b},d) := -\pi \sum_{j \neq k} d_j d_k \log|b_j - b_k| - \pi \sum_{j=1}^N d_j \mathfrak{R}_0(b_j) + \frac{1}{2}\int_{\bdry} \Phi_0(g \times g_\tau).
\end{align}
Here $g$ denotes the Dirichlet data, and $\tau$ denotes the unit tangent vector to $\partial \bD,$ and $\Phi_0$ is the unique solution to the linear problem 
\begin{align}
\label{Phi0} 
\Delta \Phi_0 &= 2\pi \sum_{j=1} d_j \delta_{b_j}, \hspace{1cm} \mbox{ in } \bD,\\
\label{eq:Phi0bc}\partial_\nu \Phi_0|_{\bdry} & = g \times g_\tau, \\
\label{eq:Phi0norm}\int_{\bdry} \Phi_0 &= 0;
\end{align}
and we define
\begin{align} \label{eq:R0}
\mathfrak{R}_0(z) := \Phi_0(z) - \sum_{j=1}^N d_j \log|z-b_j|.
\end{align}
Note that $\mathfrak{R}_0$ defines a continuous function in $\overline{\bD}.$ When $g = g_n,$ it is clear that $g \times g_\tau \equiv n.$ In this case, the formula \eqref{renorm} for the renormalized energy assumes the explicit form in \eqref{renormn} below. Using the Poisson integral formula, it can be checked that
\begin{align}\label{renormn}
W(\mathbf{b},d) = -\pi \sum_{j\neq k} d_j d_k \log|b_j - b_k| - \pi \sum_{j,k=1}^N d_j d_k \log|1-\overline{b}_j b_k|;
\end{align}
the reader can find details of a similar calculation in \cite{SandierSoret}\footnote{We thank the referee for bringing this reference to our attention.}. We point out that $W$ still depends on the degree $n$ of the boundary conditions via $\sum d_i = n.$ It is also crucial to notice in what follows that $W$ is invariant under the action of rotations in the sense that for any $\theta \in \bR,$ if we write $e^{i\theta}\mathbf{b} = (e^{i\theta}b_1,\cdots, e^{i\theta}b_N),$ then $W(\mathbf{b},d) = W(e^{i\theta}\mathbf{b},d).$ Next, we recall several periodic solutions to the point vortex flow posed in the unit disc. 
\par 
\subsection{Periodic Solutions to Point Vortex Flow}
The point vortex flow is a system of ordinary differential equations defined by 
\begin{align} \label{pvfgen}
d_j\left(\frac{\,db_j}{\,dt} \right)^\perp = - \frac{1}{\pi} \nabla_{b_j} W(b), \hspace{1cm} j=1,\cdots, N. 
\end{align}
Using \eqref{renormn}, we re-write \eqref{pvfgen} as the system 
\begin{align} \label{pvfode}
d_\ell\overline{\frac{\,db_\ell(t)}{\,dt}} = - 2i \left(\sum_{k\neq \ell} \frac{d_k}{b_\ell (t) - b_k (t)} + \sum_{k=1}^N \frac{d_k \overline{b_k(t)}}{1 - \overline{b_k(t)}b_\ell(t)} \right), \hspace{.4cm} \ell = 1,\cdots, N. \hspace{.4cm} \mathrm{(PVF) }
\end{align}
This system of equations is Hamiltonian- the phase space is $(\bD^*)^N$ and the symplectic structure is the standard one inherited from $\bC^N.$ As a consequence, the Point Vortex Flow (PVF) conserves the Hamiltonian \eqref{renormn}. In addition, this flow also conserves an angular momentum-like quantity. For any $\mathbf{b} = (b_1,\cdots, b_N) \in (\puncD)^N ,d= (d_1,\cdots d_N) \in \{\pm 1\}^N,$ define
\begin{align}
\label{eq:fdmomentum} \scrj_0(\mathbf{b},d) := -\frac{1}{2}\sum_{j=1}^N d_j |b_j|^2,
\end{align}
The following Lemma is an easy calculation.  
\begin{lemma}
Let $(\mathbf{b}(t),d) = (b_1(t), \cdots, b_N(t); d_1,\cdots, d_N)$ be a solution to \eqref{pvfode}. Then  
\begin{align}
\frac{\,d}{\,dt} \scrj_0(\mathbf{b}(t),d) = 0.
\end{align}
\end{lemma}
We will now construct some relative equilibria, i.e. uniformly rotating periodic solutions to the ODE system \eqref{pvfode}. These are solutions obtained by pursuing the ansatz 
\begin{align}
\label{eq:def rel equilibrium}
b_\ell(t) = b_\ell e^{-i\tilde{\omega} t},
\end{align}
where $b_\ell \in \bD$ is a \textit{fixed} complex number and $\tilde{\omega} \in \bR.$ Several studies have been devoted to studying such solutions for the point vortex problem on $\bC,$ see \cite{vortexcrystals} for an extensive survey. For the unit disc with boundary conditions as in \eqref{eq:bc}, the simplest such solution is that with a single ring of $n$ vortices located at the vertices of a regular $n-$gon. These are known as vortex polygons and have been known way back since \cite{thomson}. In this case, one can show (cf.\cite{qdaithesis,thomson}) that the radius of rotation $\rho$ and the angular speed $\omega_0$ are related by
\begin{align}
 \label{singleringeq}
\omega_0 \frac{\rho^2}{2} = - \left(\frac{n-1}{2} + \frac{n \rho^{2n}}{1-\rho^{2n}} \right). 
\end{align}
Observe that $\omega_0 < 0.$ Furthermore, if the momentum value $\scrj_0 = -\frac{\rho^2}{2}$ is prescribed, then $\omega_0$ is uniquely determined from \eqref{singleringeq}. 
\par 
The point vortex system posed in $\bC$ is well studied; the system corresponding to the unit disc, has received some recent attention see \cite{qdaithesis,kurakin1,speetjens} and references therein. We now give below an algebraic characterization of relative equilibria in the disc which contains as special cases results of the aforementioned studies. We closely follow a similar such program carried out in \cite{lewisratiu} for the point vortex flow posed in $\bC.$ 
\par 
In order to capture relative equilibria by symmetry, let $k > 1$ be a divisor of $n.$ We characterize relative equilibria having the $k-$fold symmetry property, i.e. those that are fixed by the group of rotations through angle $\frac{2\pi}{k}.$ Denoting by $N_-$ the number of $-1$ vortices, $N_+ := n+ N_-$ the number of $+1$ vortices and $N = N_- + N_+$ the total number of vortices, we observe easily that $N, N_{\pm}$ are all divisible by $k.$ 
\par 
Given this set up, we index the rings by Latin subscripts such as $r,s$ that range from $1,\cdots, \frac{N}{k}$ and index vortices in a given ring with a second Greek sub-script. Thus, vortices on the $r$th ring, containing $k$ vortices each of strength $d_r \in \{\pm 1\},$  are labeled $(\zeta_{r\beta})_{\beta = 0}^{k-1}.$ Then the equations \eqref{pvfode} take the form 
\begin{align} \label{pvflowrings}
d_r \overline{\frac{\,d\zeta_{r\beta}(t)}{\,dt}}   =   -2i \left( \sum_{\alpha \neq \beta} \frac{d_r}{\zeta_{r\beta}(t) - \zeta_{r \alpha}(t)} + \sum_{\alpha = 0}^{k-1} \frac{d_r \overline{\zeta_{r\alpha}(t)}}{1 - \overline{\zeta_{r\alpha}(t)} \zeta_{r \beta}(t) }  \right. \\ \notag \left. + \sum_{s \neq r} \sum_{\alpha=0}^{k-1} \frac{d_s}{\zeta_{r\beta}(t) - \zeta_{s \alpha}(t)} + \sum_{s \neq r} \sum_{\alpha=0}^{k-1} \frac{d_s \overline{\zeta_{s\alpha}(t)}}{1 - \overline{\zeta_{s\alpha}(t)} \zeta_{r \beta} }\right).
\end{align}
Requiring $k-$fold symmetry, we pursue the ansatz
\begin{align*}
\zeta_{r\beta}(t) = e^{i\left(\frac{2\pi \beta}{k}\right)} \zeta_r(t).
\end{align*}
Plugging in this ansatz into \eqref{pvflowrings} and computing using the properties of $k$-th roots of unity, we find,
\begin{align}
d_r \zeta_r(t)\overline{ \zeta_r'(t)} = -2i \left(  \frac{k-1}{2} + \frac{d_r k |\zeta_r(t)|^{2k}}{1- |\zeta_r(t)|^{2k}} + k \sum_{s \neq r} d_s \zeta_r^k(t) \left \{ \frac{1}{\zeta_r^k(t) - \zeta_s^k(t)} + \frac{\overline{\zeta_s^k(t)}}{1- \overline{\zeta_s^k(t)} \zeta_r^k(t) }\right\}\right).
\end{align}
Seeking relative equilibria, we further suppose 
\begin{align*}
\zeta_r (t) = e^{-i\omega_0 t} \zeta_r,
\end{align*}
where $\zeta_r \in \bD.$ We obtain the following algebraic characterization of the relative equilibrium 
\begin{align} \label{eq:rings eq}
-d_r \omega_0 |\zeta_r|^2 = 2 \left( \frac{k-1}{2} + \frac{d_r k |\zeta_r|^{2k}}{1- |\zeta_r|^{2k} } + k \sum_{s \neq r} d_s\zeta_r^k \left\{\frac{1}{\zeta_r^k - \zeta_s^k} + \frac{\overline{\zeta_s^k}}{1 - \overline{\zeta_s^k} \zeta_r^k} \right\} \right).
\end{align}
This last expression holds for each $r = 1, \cdots, \frac{N}{k}.$ Thus for each $r,$ it is necessary that 
\begin{align} \label{algconditions}
\sum_{s \neq r} d_s \zeta_r^k  \left\{\frac{1}{\zeta_r^k - \zeta_s^k} + \frac{\overline{\zeta_s^k}}{1 - \overline{\zeta_s^k} \zeta_r^k} \right\} \in \bR. 
\end{align}
Writing $\zeta_s = \rho_s e^{i\phi_s},$ we introduce $\mu_{sr} = \frac{\rho_s}{\rho_r}, \nu_{sr} = \rho_s \rho_r,$ and $\psi_{sr} = \phi_s - \phi_r.$ Inserting this in the last equation, we find that \eqref{algconditions} is equivalent to 
\begin{align} \label{eq:angle conditions}
\sum_{s \neq r} d_s \sin(k\psi_{sr}) \left( P_{sr}(\mu_{sr}^k) - P_{sr}\left(\frac{1}{\nu_{sr}^k} \right)\right) = 0, 
\end{align}
with 
\begin{align}
P_{sr}(t) := \frac{2t}{1-2t \cos(k \psi_{sr}) + t^2}
\end{align}
holding for each $r = 1, \cdots, \frac{N}{k}.$ 
\par 
The conditions \eqref{eq:angle conditions} are satisfied provided 
\begin{align} \label{eq:angle conditions 2}
\sin(k \psi_{sr}) = 0,
\end{align}
or in other words, that $\phi_s - \phi_r = 0$ or $\frac{\pi}{k}$ (only $\mod \frac{2\pi}{k}$ matters because of symmetry \eqref{kfoldsym}). We have not pursued the necessity of \eqref{eq:angle conditions 2} in order that \eqref{eq:angle conditions} holds,  but it appears to be true for small $k,n.$ 
\par 
If for each pair of distinct $s,r$ one has $\phi_s - \phi_r = 0,$ the relative equilibrium is said to be an \textit{aligned} configuration. In such an arrangement, the vortices lie on the same "rays" emanating out of the origin (cf. Figure 1). 
\begin{figure}[H] 
 \centering
     \label{fig:aligned}
\begin{tikzpicture}
\begin{axis}[
    axis lines=middle,
    xmin=-10, xmax=10,
    ymin=-10, ymax=10,
    xtick=\empty, ytick=\empty
]
\addplot [only marks] table {
2.5 0
1.25  2.165
-1.25  2.165   
-2.5  0
-1.25  -2.165
1.25   -2.165
};
\addplot [only marks, mark=o] table {
5 0
2.5  4.33
-2.5  4.33   
-5  0
-2.5  -4.33
2.5   -4.33
};
\draw (axis cs:0,0) circle [blue, radius= 75];
%\addplot [domain=-10:10, samples=2, dashed] {1*x+3};
\end{axis}
\end{tikzpicture}
\caption{An aligned configuration. The solid and hollow bullets indicate possibly different degrees. $k = 6.$ Not to scale. }
\end{figure}
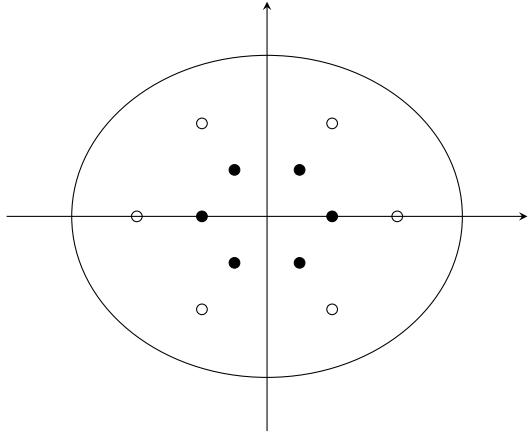
 \begin{figure} \label{fig:staggered}
        \centering
        \begin{tikzpicture}
\begin{axis}[
    axis lines=middle,
    xmin=-10, xmax=10,
    ymin=-10, ymax=10,
    xtick=\empty, ytick=\empty
]
\addplot [only marks] table {
2 0
0  2
-2  0  
0  -2
};
\addplot [only marks, mark=o] table {
3.535 3.535
-3.535  3.535
-3.535  -3.535   
3.535 -3.535
};
\draw (axis cs:0,0) circle [blue, radius= 75];
%\addplot [domain=-10:10, samples=2, dashed] {1*x+3};
\end{axis}
\end{tikzpicture}
\caption{An staggered configuration. The solid and hollow bullets indicate possibly different degrees. $k = 4.$ Not to scale. }
\end{figure}
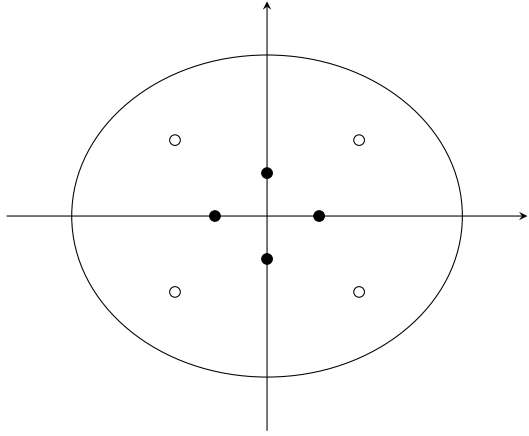
If on the other hand, $\phi_s - \phi_r = \frac{\pi}{k},$ for some $s,r$ the resulting arrangement of vortices is said to be a \textit{staggered} or \textit{alternate} configuration (cf. Figure 2 ). 
\par
Returning to \eqref{eq:rings eq} and fixing values of $\psi_{sr}$ results in a system of coupled algebraic equations with ring radii as unknowns. At this stage, the authors in \cite{lewisratiu} exhibit numerical examples in which they can solve the resulting system of algebraic equations. It is however important to note that vortex strengths can be arbitrary real numbers in \cite{lewisratiu}, whereas $d_r \in \{\pm 1\}$ for our purposes. %Except for the few examples that we collect below, we have not strenuously attempted the difficult problem of finding further relative equilibria for (PVF) in the disc for the following reason: as long as $(a_1,d_1),\cdots, (a_N,d_N)$ is \textit{any} relative equilibrium to (PVF) in the disc, with $\{d_j\}\in \{\pm 1\},$ and satisfies certain symmetry conditions explained in Section  \ref{rem:count} below, our Theorems \ref{thm:existence}-\ref{thm:asymptotics} yield periodic solutions to \eqref{eq:GP}-\eqref{eq:bc} that follow this relative equilibrium for all time. 
Here are a few cases where one can solve the resulting system of algebraic equations. 
\begin{enumerate}
\item Suppose $N_- = 0,$ so $N = N_+ = n = km.$ Further suppose that $\psi_{sr} = 0.$ In other words we have $m$ rings of $k$ vortices each, all rings being aligned. In this case, for any $p \in (\frac{-n}{2},0)$ there exist ring radii $0 < \varrho_1 < \varrho_2 < \cdots < \varrho_m < 1$ solving the preceding system of algebraic equations. The case $m = 1$ is immediate and is \eqref{singleringeq}. The case $m =2$ can be found in \cite{qdaithesis}, and the general case follows by induction on $m.$  Further the angular velocity in this case is given by adding the equations \eqref{eq:rings eq}. One obtains 
\begin{align*}
-\omega_0 (\varrho_1^2 + \cdots \varrho_m^2) = 2 \left( m\frac{k-1}{2} + k \binom{m}{2} + k \sum_{r,s =1}^m \frac{(\varrho_r\varrho_s)^k}{1 - (\varrho_r \varrho_s)^k} \right). 
\end{align*} 
Observe once again that $\omega_0 < 0.$ Furthermore, setting $m = 1,$ along with the convention that the binomial coefficient $\binom{1}{2} = 0$ yield the classical solution \eqref{singleringeq}. 
\item Suppose $m = 2$ and $\psi_{12} = \frac{\pi}{k}.$ In this case define $r_0(k) = (\sqrt{4k^2 + 1}-2k)^{1/2k}.$  If $d_1 = 1$ and $d_2 = -1,$ then for any $r <r_0(k)$ there exist $0 < \varrho_1 < r < \varrho_2 < 1$ such that $k$ degree $+1$ vortices located at radius $\varrho_1,$ and $k$ degree $-1$ vortices located at radius $\varrho_2$ in a ring skewed relative to the first one, form a uniformly rotating relative equilibrium. On the other hand, if $d_1 = -1$ and $d_2 = 1$ then for any $r > r_0$ there exist $0 < \varrho_1 < r < \varrho_2 < 1$ such that the same is true. 
\item The special case with $\varrho_1 = \varrho_2$ is a relative equilibrium if and only if $\varrho_1 = \varrho_2 = r_0.$ In this case 
\begin{align*}
\omega_0 = \frac{4k + 1 - \sqrt{4k^2 + 1}}{8\pi (\sqrt{4k^2 + 1} - 2k)^{1/2k}}(d_1 + d_2). 
\end{align*}
In particular if we have that $d_1 + d_2 = 0$ then $\omega_0 = 0.$ In other words, it is a critical point of the re-normalized energy $W$. 
\end{enumerate}
We close this section by remarking adding the system of equations \eqref{eq:rings eq} in $r$ and using the definition of $\scrj_0$ yields the following expression for the angular speed $\omega_0:$ 
\begin{align*}
\omega_0 = \frac{1}{\scrj_0(\zeta;d)} \left( \frac{N}{k} \cdot \frac{k-1}{2} + k\sum_{r =1}^{N/k} \frac{d_r |\zeta_r|^{2k}}{1 - |\zeta_r|^{2k}} + k \sum_{r,s; r\neq s} d_s \zeta_r^k \left\{ \frac{1}{\zeta_r^k - \zeta_s^k} + \frac{\overline{\zeta_s^k}}{1 - \overline{\zeta_s^k} \zeta_r^k}\right\}\right). 
\end{align*} 

\section{Periodic Solutions to Gross-Pitaevskii via Constrained Minimization}
 \label{sec:Min}
In this section we will confine our attention to constructing periodic solutions to \eqref{eq:GP} that only carry vortices of degree $+1,$ (or $-1,$ but not both) and that follow corresponding solutions to the point vortex flow. Indeed, the strategy in this section is to focus on constrained minimizers of $E_\e,$ which can only possess vortices of one sign. The main result of this section is contained in Theorem \ref{gpsoltheorem}; we begin with some preliminary results. 
\begin{proposition} \label{prop1}
Let $p \in \bR.$ Assume $k$ is a positive integer that divides $n.$  Then constrained smooth critical points of the Ginzburg Landau energy $E_\e$ defined in \eqref{gl} among competitors in the set 
\begin{align}
\label{admissible}
\mathcal{A}^{n,k}_p := \{ v \in H^1_{g_n}(\bD,\bC), v\big(e^{i(\frac{2\pi}{k})}z \big) = v(z), \scrj(v) = p    \}
\end{align}
satisfy the elliptic PDE \eqref{ansatzpde} with $\omega_\e$ arising as a Lagrange multiplier. 
\end{proposition}
\begin{proof}
The proof requires computing the first variations $\delta E_\e(v)$ and $\delta\scrj(v)$ of the Ginzburg-Landau Energy \eqref{gl} and the constraint $\scrj$ defined in \eqref{momconstraint} respectively. For any $w \in H^1_0(\Omega)$ we compute, 
\begin{align*}
\delta E_\e(v)w = \frac{\,d}{\,d s} \Big\vert_{s=0} E_\e(v + sw) =  \int_{\bD} \nabla v \nabla w - \frac{v \cdot w}{\e^2}(1-|v|^2),
\end{align*}
and 
\begin{align*}
\delta \scrj(v)w = \frac{\,d}{\,d s} \Big\vert_{s=0} \scrj(v + sw) = - \int_{\bD} k v \cdot w + \frac{1}{m} w \cdot (y^\perp \nabla)v^\perp \,dy.
\end{align*}
By the Lagrange multiplier theorem, there exists $\omega_\e \in \bR$ such that constrained critical points $v$ satisfy 
\begin{align*}
\delta E_\e(v)w = \omega_\e \delta \scrj (v) w.
\end{align*}
The result of the proposition follows by integrating by parts once on the left-hand side. We remark that the symmetry conditions in the definition of the admissible class do not alter the Euler-Lagrange equations, as can be easily checked using the weak formulation. 
\end{proof}
\par 
Suppose next that $k,p$ are such that the set $\mathcal{A}^{n,k}_p \neq \emptyset.$ We next show that the Direct Method in the calculus of variations yields existence of a constrained minimizer of $E_\e$ within this class. 
\begin{proposition} \label{prop2directmethod}
Let $p \in \bR,$ suppose $k$ is a divisor of $n,$ and we have that the admissible set $\mathcal{A}^{n,k}_p$ defined in \eqref{admissible} is nonempty. Then, for each $\e >0,$ there exists a minimizer $v$ of the Ginzburg Landau energy $E_\e$ within the admissible class $\mathcal{A}^{n,k}_p.$ Moreover, $v$ is a smooth solution to the PDE \eqref{ansatzpde}.
\end{proposition}
\begin{proof}
Suppose $\{v_k\}$ is a minimizing sequence in $\mathcal{A}^{n,k}_p \neq \emptyset.$ Then, it is clear that $\{v_k\}$ is bounded in $H^1,$ and hence by Alaoglu's theorem, converges weakly (upon extraction, not denoted) to $v \in H^1(\bD, \bC).$  By Rellich's theorem, a subsequence of $\{v_k\}$ (not relabeled) converges strongly in $L^2$ to $v.$  Using standard facts from weak convergence, we conclude that $\scrj(v_k) \to \scrj(v) = p.$ Finally, by Mazur's theorem, $v$ has the same boundary conditions as $v_k;$ and furthermore, inherits the symmetry property of $v_k.$ Thus $v$ is a constrained minimizer of the Ginzburg Landau energy, with $\scrj(v) = p.$ Since the constraint functional $\scrj$ is merely quadratic so that $\delta \scrj$ is a bounded linear functional on $H^1,$ standard arguments then show that $v$ is a weak solution to the associated constrained Euler-Lagrange equations with an unknown Lagrange multiplier. Elliptic regularity then implies that $v$ is smooth up to the boundary. Consequently, by  Proposition \ref{prop1}, the desired conclusion follows. 
\end{proof}
We will first study single ring vortex polygons and obtain time periodic solutions to Gross Pitaevskii that follow these solutions to (PVF). The case of multiple ring solutions is harder via constrained minimization, and we are only able to show the existence of multiple ring time periodic solutions to \eqref{eq:GP}, but are unable to show that these follow corresponding relative equilibria to (PVF). We describe this partial result towards the end of this section. 
\subsection{Single Ring Solutions}
Given a vortex polygon as in \eqref{singleringeq} above, we now construct a sequence of solutions $v_\e$ to \eqref{ansatzpde} for $\e > 0$ small, whose vortices follow this vortex polygon as $\e \to 0.$ We note that for such single ring solutions we use $m = 1,$ i.e. $k = n$ in the definition of $\scrj.$  
\begin{proposition} \label{existenceminsingle}
Let $p \in (-\frac{n}{2},0).$ There exist a sequence of constraint values $p_\e \to \pi p$ as $\e \to 0$ and a corresponding sequence of constrained minimizers $v_\e$ of $E_\e$ within the class $\mathcal{A}^{n,k}_{p_\e}$ satisfying 
\begin{align} \label{logbound}
E_\e(v_\e) \leq n \pi \log \frac{1}{\e} + O_\e(1)
\end{align}
as $\e \to 0.$ 
\end{proposition}
\begin{proof}
Fix $p \in (-\frac{n}{2},0),$ and select $r \in (0,1)$ such that 
$\frac{n  r^2}{2} = -p. $ Define then points $b_j = r e^{i\big(\frac{2\pi j}{n}\big)},$ and write $b = (b_1,\cdots, b_n).$ Recall from \eqref{eq:fdmomentum} that 
\begin{align*}
\scrj_0(b,d) = -\frac{1}{2} \sum_{i=1}^n |b_i|^2 = -n\frac{r^2}{2} = p.
\end{align*}
Let $\tilde{v}_\e$ be a function as in Proposition \ref{prop:BBH construction}, satisfying 
\begin{align}
E_\e(\tilde{v}_\e) = n\left( \pi \log \frac{1}{\e} + \gamma \right) + W(b) + o_\e(1),
\end{align}
and 
\begin{align}
|\scrj(\tilde{v}_\e) - \pi p| = o_\e(1) 
\end{align}
as $\e \to 0.$ Define $p_\e := \scrj(\tilde{v}_\e).$ Then $\tilde{v}_\e \in \mathcal{A}^{n,n}_{p_\e},$ so the latter is not the empty set. Consequently, by Proposition \ref{prop2directmethod}, we conclude that there exists a smooth minimizer $v_\e$ with $\scrj(v_\e) = p_\e,$ and by comparison with the energy of $\tilde{v}_\e,$ we have \eqref{logbound}.
\end{proof}
We are now ready for the main result of this section. This uses in a crucial way the vortex balls construction (cf. \cite[Theorem 4.1]{serfatysandierbook}). 
\begin{theorem} \label{gpsoltheorem}
Let $b(t)$ be a single vortex polygon solution to \eqref{pvflowrings} in the unit disc, with vortices located on a circle $C$ of radius $\rho,$ related to the angular speed $\omega_0$ of the flow by \eqref{singleringeq}. Then for all sufficiently small $\e>0$ there exists a $\frac{2\pi}{|\omega_\e|}$-periodic solution $u_\e(t,x) = R(-n\omega_\e t) v_\e(R(\omega_\e t)x)$ (cf. \eqref{ansatz}) satisfying
\begin{align*}
u_\e(r,\theta + \frac{2\pi}{n},t) = u(r,\theta,t)
\end{align*}
for all time $t \in \bR,$ and verifying the energy estimate $E_\e(u_\e) \leq \pi n |\log \e| + O_\e(1)$ as $\e \to 0.$ Furthermore, there exist a finite collection of disjoint balls $\mathscr{B}_\e$ in $\bD$ including $n$ balls $ B^j := e^{i(\tfrac{2\pi j}{d})}B^0_{r_\e}$ for $j=0,\cdots, n-1$ such that 
\begin{enumerate}
\item The radii $r_\e$ converge to $0$ as $\e \to 0,$ 
\item Outside the collection $\mathscr{B}_\e,$ we have that $|u_\e| > 1/2,$ and so $\deg(u_\e,\partial B)$ is well defined for each ball $B \in \mathscr{B}_\e.$ Moreover, $\deg(u_\e, \partial B^j) = 1,$ and $\deg(u_\e, \partial B) = 0$ for any $B \in \mathscr{B}_\e \backslash \bigcup_j B^j,$ 
\item $|\mathscr{B}_\e| = o_\e(1)$ as $\e \to 0.$ 
\item The common circular orbit $C^\e$ associated with the rotation of the centers of the balls $B^j$  approaches $C$ as $\e \to 0.$ 
\end{enumerate}
\end{theorem}
\begin{proof}
Let $b(t)$ be a vortex polygon solution to \eqref{pvfode}. Then $b(t) = (b_0(t), b_1(t),\cdots, b_n(t)),$ where $b_j(t) = re^{i(-\omega_0 t)} e^{i(\frac{2\pi j}{n})},$ and $r $ and $\omega_0$ are related by \eqref{singleringeq}. Set $p := -\frac{n r^2}{2}.$ Then by Proposition \ref{existenceminsingle}, for sufficiently small $\e >0,$ there exist a sequence of solutions $v_\e$ to \eqref{ansatzpde} satisfying the bound 
\begin{align*}
E_\e(v_\e) \leq n \pi \log \frac{1}{\e} + O_\e(1),
\end{align*}
and 
\begin{align*}
|\scrj(v_\e) - \pi p| = o_\e(1)
\end{align*}
as $\e \to 0.$ Setting $u_\e(t,x) = R(-n\omega_\e t) v_\e(R(\omega_\e t)x)$ proves the existence of periodic orbits to GP with the desired symmetry property and energy bound. It remains to prove items (1-4) in the statement of the Theorem. The logarithmic energy bound on the functions $v_\e$ permits us to invoke the vortex balls construction \cite[Theorem 4.1]{serfatysandierbook}, which when applied to $v_\e$ with $\alpha = \frac{1}{2},$ implies that there exist a finite collection of disjoint closed balls $\mathcal{B}_\e = \{B^\e_j\}_{j=1}^{N_\e}$ with $\sum_{j=1}^{N_\e} r_j^\e \lesssim \e^{1/4},$ where $r_j^\e$ is the radius of the ball $B^\e_j,$ and 
\begin{align*}
\{x \in \bD: \big||v_\e(x)| - 1 \big| \geq \e^{1/8} \} \subset \bigcup_{j=1}^{N_\e} B^\e_j.
\end{align*} 
This implies items (1) and (3), and also that $|v_\e| > 1/2$ outside the collection $\mathcal{B}_\e.$ Set $D_\e = \sum_{j=1}^{N_\e} |d^\e_j|,$ with $d^\e_j$ as in the statement of \cite[Theorem 4.1]{serfatysandierbook}. Then it can be seen that $D_\e > 0,$ and $d^\e_j \equiv +1.$  The proof of this fact uses the observation if $D_\e = 0$ for some $\e,$ then the momentum $\scrj(v_\e)$ is forced to be vanishingly small, while if $d^\e_j \neq +1$ for some $j,$ then the Ginzburg-Landau energy $E_\e$ is too big. Since the details are very similar to Theorem 4.5 of \cite{GS} we omit them here. 
\par The facts that $D_\e > 0$ and $d^\e_j \equiv +1,$ combined with the lower energy bound supplied by the vortex balls construction, cf. \cite[Eqn. (4.4), Theorem 4.1]{serfatysandierbook} and the symmetry of $v_\e$  shows item (2) is true. It remains to show (4). Denote the centers of the balls $B^j$ mentioned in the statement of the Theorem by $a_\e^j.$ Then, similar to the proof of \eqref{momentumoptimal} in Proposition \ref{prop:BBH construction} we can argue that $\scrj(v_\e) = -\frac{n\pi}{2} \sum_{j=1}^{n} |a_\e^j|^2 +o_\e(1)$ as $\e \to 0.$ But since $\scrj(v_\e) = p_\e = \pi p + o_\e(1),$ and since $|a_\e^j | = |a_\e^0|$ for all $j,$ it follows that $|a_\e^j - re^{i(\frac{2\pi j}{n})}| = o_\e(1)$ as $\e \to 0.$ This completes the proof of item (4), and also that of the theorem.  
\end{proof}
It is natural to hope to be able to relate the speeds $\omega_\e$ of the time periodic solutions $v_\e$ to \eqref{eq:GP} that we have constructed, to the speed $\omega_0$ of (PVF) given by \eqref{singleringeq}. This seems to require (and in fact, to be equivalent to) a bound on the potential term of the form 
\begin{align*}
\frac{1}{\e^2}\int_{\bD} (1- |v_\e|^2)^2 \leq C.
\end{align*}
The general starting point for an estimate like this is a Pohazaev identity, cf. Lemma \ref{lem:Pohazaev} below. However, in the present case, this approach fails since we do not have control on $\omega_\e$ independent of $\e.$ 
\subsection{Remarks on Multiple Ring Solutions}
In this section, we briefly show existence of solutions to Gross Pitaevskii that have multiple rings of $+1$ vortices that are aligned. Regrettably however, we are as yet unable to relate these solutions to corresponding solutions of point vortex flow in a very concrete manner. The difficulty is the same as earlier: we do not have the much-sought after uniform bound on the potential term. 
\par 
Once again, let $\mathbf{a} = (a_1,\cdots, a_n)$ denote a relative  equilibrium with $d_i \equiv 1$ for all $i = 1,\cdots, n.$ Let $k$ divide $n$ and suppose the $a_1,\cdots a_n$ are arranged in $\frac{n}{k}$ rings of radii $\rho_1, \cdots, \rho_{n/k}$ respectively, each containing $k$ vortices each. Suppose further that the rings are aligned. Thus 
\begin{align*}
\scrj_0(a_1,\cdots,a_n) = -\frac{k}{2}(\rho_1^2 + \cdots + \rho_{n/k}^2). 
\end{align*}
The strategy is to find constrained minimizers of $E_\e(u)$ within the admissible set with further symmetry assumptions that ensure that the rings are aligned. Specifically, define
\begin{align}
\mathfrak{B}^{n,k}_p := 
\left\{ u \in H^1_{g_n} (\bD, \bC)\big| \left. \begin{array}{cc} u\left(r,\theta + \frac{2\pi}{k}\right) &= u(r,\theta), 0 \leq r \leq 1, \theta \in [0,2\pi) \\   u\left(r, \frac{\pi j}{k} + \theta \right) &= \overline{u}\left(r,\frac{\pi j}{k} - \theta\right),  j \mbox{ odd }, \theta \in [0, \frac{\pi}{k}]\end{array}\right. \right\}. 
%  \begin{array}{cc} \\ \end{array}\right.\right\} \right\}
\end{align}
Furthermore, it can also be checked easily that the constraints imposed above do not alter the Euler Lagrange equation \eqref{ansatzpde}. We then note that Proposition \ref{prop2directmethod} applies and yields a smooth constrained minimizer for constraint values $p$ such that $\mathfrak{B}^{n,k}_p \neq \emptyset.$  Then, given $p \in (-\frac{n}{2},0),$ we fix a solution to (PVF) with this $p-$value, noting that such solutions are not necessarily unique. Proceeding as in Proposition \ref{existenceminsingle}, we obtain a sequence $p_\e \to p$ as $\e \to 0$ and corresponding sequence of constrained minimizers $v_\e \in \mathfrak{B}^{n,k}_{p_\e}$ satisfying an energy bound 
\begin{align*}
E_\e(v_\e) \leq n \pi \log \frac{1}{\e} + O_\e(1)
\end{align*} 
as $\e \to 0.$ However, contrary to the single ring case, we can \textit{not} obtain any connection between vortex locations of $v_\e$ and those corresponding to solutions of (PVF). Indeed, in the single ring case, if $\varrho$ is the radius of the ring, then prescribing a constraint value for $\scrj_0(b) = - \frac{\varrho^2}{2},$  fixes $\varrho.$ In the case of multiple rings on the other hand, a condition of the form 
\begin{align*}
\varrho_1^2 + \cdots + \varrho_{n/k}^2 = p
\end{align*}
does not suffice to determine $\varrho_1, \cdots, \varrho_{n/k}.$ Nevertheless, using the vortex balls theorem and the Jacobian estimate, we \textit{can} argue that $Jv_\e \rightharpoonup \pi \sum_{i=1}^n d_i \delta_{c_i}$ for points $c_i \in \bD$ that have the same symmetry properties as the $\{a_i\}.$ What we \textit{cannot} argue, is that the collection of points $\{c_i\}$ is the same as the collection $\{a_i\}$ or even that it is a relative equilibrium of (PVF). 

\section{Periodic Solutions to Gross-Pitaevskii via Linking} \label{sec:Lin}
In this section, we expand on the results of the preceding section, using a completely different approach that is topological. Given a relative equilibrium for the point vortex flow with time period $T,$ for $\e > 0$ sufficiently small, we construct periodic solutions to the Gross Pitaevskii equation that have the \textit{same time period} $T,$ and follow the given relative equilibrium to point vortex flow as $\e \to 0.$ The approach of this section also permits the analysis of multi-ring solutions with vortices of degrees $\pm 1$ coexisting. 
\par 
We present the results of this section in two main theorems, viz. Theorems \ref{thm:existence} and \ref{thm:asymptotics}. Theorem \ref{thm:existence} concerns itself with existence of periodic solutions to Gross Pitaevskii for sufficiently small $\e,$ while Theorem \ref{thm:asymptotics} concerns itself with $\e \to 0$ asymptotics. For the convenience of the reader, we begin by informally sketching the ideas involved in the proofs of these theorems. \par
\textbf{Sketch of the ideas in the proof:}
Existence of periodic orbits to GP corresponding to a given relative equilibrium is proved by obtaining critical points of the functional $\cale^{\omega_0},$ defined by 
\begin{align}
\label{scripte}
\cale^{\omega_0} (u) := \int_\bD \frac{|\nabla u|^2}{2} + \frac{1}{4\e^2} (1-|u|^2)^2 + \frac{\omega_0}{2} \int_\bD k |u|^2 + \frac{1}{m} u \cdot (y^\perp \cdot \nabla )u^\perp \,dy,
\end{align}
where $\omega_0$ is the angular speed of rotation of the given relative equilibrium of (PVF), cf. formula \eqref{eq:period formula}. Indeed, the system \eqref{ansatzpde} with $\omega = \omega_0$ represents criticality conditions of the functional $\cale^{\omega_0} \equiv E_\e - \omega_0 \scrj.$ We find critical points of $\scre^{\omega_0}_\e$ as large time limits of the gradient flow of $\scre^{\omega_0}_\e,$ and adapt a localized topological linking procedure from \cite{LinMinMax} to retain control on vortex locations; this is the content of Theorem \ref{thm:existence}. The heart of the matter in its proof is that the relative equilibrium $\mathbf{a} := (a_1, \cdots, a_N)$ given in the Theorem yields a critical point, also denoted $\mathbf{a}$, of a modified renormalized energy functional $\calh^{n,\omega_0} := \frac{1}{\pi} W - \omega_0 \scrj_0.$ The positive and negative eigenvalues of $D^2\calh^{n,\omega_0}(\mathbf{a})$ yield a topological linking structure in a neighborhood of $\mathbf{a}$; the construction in Prop.\ref{prop:BBH construction} transports this linking structure to the Sobolev space $H^1,$ yielding a rich collection of initial data to our gradient flow. The gradient flow structure yields upper bounds on the energy $\scre^{\omega_0}_\e$ along the flow. Control on the vortex locations during the course of the gradient flow is given by a projection lemma cf. Prop. \ref{lem: proj lemma}, which then along with the linking structure yields matching lower bounds on the $\scre^{\omega_0}_\e$ energy. The gradient flow being a homotopy, preserves the linking structure, and yields a critical point of $\scre^{\omega_0}_\e$ in the $t \to \infty$ limit.  An interesting technical ingredient of this proof is a symmetry argument which asserts that the gradient flow of $\scre^{\omega_0}_\e$ preserves the symmetry of the initial data; its proof is an elementary argument involving Fourier series, ODE uniqueness and divisibility properties of integers, cf. Claim 4.7. Theorem \ref{thm:asymptotics}, which proves the $\e \to 0$ asymptotics of the solutions constructed in Theorem \ref{thm:existence} involves a number of integral identities, and follows the general program laid out in \cite{BBH}.
\par
We start with some preparation before stating the main theorems of this section. 
\par 
\begin{lemma} \label{lem:Palais-Smale}
Let $\omega \in \bR$ and $\e > 0.$ The functional $\cale^{\omega}$ satisfies the Palais-Smale condition in the space $H^1 (\bD, \bC).$ 
\end{lemma}
We omit the proof of this Lemma. since its proof is identical to the fact that the Ginzburg Landau energy verifies the Palais Smale condition \cite{jerrardsternbergcriticalpoints}.
\par
\begin{definition} \label{rem:generating set}
Let $\{a_1,\cdots, a_N\}$ be a given set of vortices with the $k$-fold symmetry property. A subset $S \subset \{a_1,\cdots,a_N\} = \mathbf{a}$ is said to be a \textit{generating set} of $\mathbf{a}$ if, recalling the notation in \eqref{eq:multipleofset}, the disjoint union 
\begin{align*}
\bigsqcup_{j=0}^{k-1} e^{\left(\frac{2\pi j}{k}\right)} S = \mathbf{a}.
\end{align*}
Essentially, selecting a generating set for a multi-ring configuration amounts to selecting a representative from each ring. Now, let $N_R = \frac{N}{k}$ denotes the number of rings, and  $S = \{a_1, \cdots, a_{N_R}\}$ be a generating set. By considerations of symmetry, we may assume that $a_j \neq 0$ for any $j.$ For $j = 1,\cdots N_R$ let $\theta_j := \arg a_j.$ Then we can make the choice of the generating set so that $\max_{i \neq j} |\theta_i - \theta_j| \leq \frac{2\pi}{k}.$ For the rest of the paper, given $\mathbf{a} \in (\bD^*)^N,$ we will denote by $\wma \in (\bD^*)^{N_R}$ such a choice of generating set. 
\end{definition}
As a convenient abuse of notation, we will frequently use 
\begin{align} \label{eq:convention 2}
\calh^{n,\omega_0} (\wma,d) := \calh^{n,\omega_0} (\mathbf{a},d).
\end{align}
Since the functional $\calh^{n,\omega_0}$ is not affected by re-labeling of the pairs $(a_i,d_i),$ this definition is well-defined. 

%\begin{remark}
Our next definition is a finite dimensional version of $\scre^{\omega_0}_\e.$  Suppose $\mathbf{b} = (b_1,\cdots, b_N) \in (\bD^*)^N$ and $d = (d_1,\cdots, d_N) \in \{\pm 1 \}^N$ with $\sum d_i = n.$ In order for $(\mathbf{b},d)$ to be a relative equilibrium of the point vortex flow, cf. \eqref{eq:def rel equilibrium} with angular speed $\frac{\omega_0}{m},$ we require that $\mathbf{b}$ is a critical point of the functional 
\begin{align}
\label{eq:calh}
\mathcal{H}^{n,\omega_0} (\mathbf{b},d) := \frac{1}{\pi} W(\mathbf{b},d) - \frac{\omega_0}{m} \scrj_0(\mathbf{b},d), 
\end{align}
with $\scrj_0$ defined as in Eq. \eqref{eq:fdmomentum}. We note that when $\omega_0 \neq 0,$ the time-period of rotation of these periodic solutions to (PVF) is given by 
\begin{align} \label{eq:period formula}
T = \frac{2\pi m}{|\omega_0|}. 
\end{align}
\begin{remark}
When the degree $d = (d_1, \cdots, d_N)$ has been fixed, we will frequently \textit{drop} the dependence on $d$ from our notation, but continue to denote by $\calh^{n,\omega_0}(\mathbf{b}) = \calh^{n,\omega_0}(\mathbf{b},d)$ as a convenient abuse of notation.  
\end{remark}

Throughout the course of the proofs of Theorems \ref{thm:existence} and \ref{thm:asymptotics} we make the following
\par 
\begin{em}\textbf{Assumption A:}\\
By rotational invariance of the functional $\calh^{n,\omega_0},$ it follows that for every $\theta \in \bR$ the point $(e^{i\theta}\mathbf{a} ,d)$ is critical for the functional $\calh^{n,\omega_0}.$ We introduce the smooth curve in $\bR^{2N_R}$ defined by
\begin{align*}
\call_{\wma} := \{ e^{i\theta}\wma : 0 \leq \theta < 2\pi \}.
\end{align*}
It is then clear that a tangent vector to $\call_{\wma}$ is a null-vector for $D^2 \calh^{n,\omega_0}.$ We will assume that this is the only null-direction.
\end{em}
\par
This assumption can be substantially relaxed; we explain this along with necessary modifications to the proofs of Theorems \ref{thm:existence}-\ref{thm:asymptotics} in Remark \ref{rem:modification} below. However, in an attempt to avoid clutter, we present the proofs of our Theorems under Assumption A. 
\par
\begin{theorem} \label{thm:existence}
Let $(\mathbf{a}(t),d)$ be a time periodic solution to the point vortex flow \eqref{pvfode} with time period $ \frac{2\pi m}{\omega_0}$,cf. Eq. \eqref{eq:period formula}. Denote by $(\mathbf{a},d)$ the corresponding critical point to the functional $\calh^{n,\omega_0}$ defined in \eqref{eq:calh}. Also let $\mathbf{a}$ satisfy Assumption A above. Then there exists $\e_0 > 0$ depending only on $n$ and the critical point $\ma$ such that the following holds: for all $0 < \e < \e_0$ there exists a $\frac{2\pi m}{\omega_0}$-time periodic vortex solution $u_\e(x,t) := R(-k \omega_0 t) v_\e(R\big(\frac{\omega_0}{m}t\big) x)$ of \eqref{eq:GP} satisfying 
\begin{align} \label{eq:symmetry prop of solution}
u_\e(x;t) = u_\e(r,\theta;t) = u_\e\left(r,\theta + \frac{2\pi}{k};t\right)
\end{align}
for all $r \in (0,1], \theta \in [0,2\pi)$ and $t \in \bR.$ Furthermore, the function $v_\e = v_\e(y) : \bD \to \bR^2$ is a smooth critical point of the functional $\cale^{\omega_0}$ defined in \eqref{scripte}, has the $k-$fold symmetry property, and satisfies the energy estimate 
\begin{align} \label{eq:energy close}
\left| \cale^{\omega_0} (v_\e) - N \left( \pi \log \frac{1}{\e} + \gamma \right) - \pi\calh^{n,\omega_0} (\ma) \right| = o_\e(1), \hspace{1cm} \mbox{ as } \e \to 0, 
\end{align}
where $\gamma \in \bR$ is a universal constant defined below in \eqref{eq:gamma}. 
\end{theorem}
\begin{remark}
We will prove that the time periodic solutions constructed in Theorem \ref{thm:existence} do indeed follow the given solution of (PVF) in Theorem \ref{thm:asymptotics}. 
\end{remark}
\begin{proof}[Proof of Theorem \ref{thm:existence}]
The proof of this theorem is organized in a sequence of steps. Our goal in much of the proof is to build a critical point of $\cale^{\omega_0}$ introduced in \eqref{scripte}. \\
\textbf{Step 1:} We begin by localizing near a critical point of $\calh^{n,\omega_0}.$ Let $\ma := (a_1, \cdots , a_N)$ be a critical point of $\calh^{n,\omega_0}$ corresponding to the given periodic solution of \eqref{pvfode} with degrees $d_1,\cdots, d_N$. Then, as remarked before, for each $\theta \in \bR,$ we have the invariance
\begin{align*}
\calh^{n,\omega_0} (a_1,\cdots, a_N, d_1,\cdots, d_N) = \calh^{n,\omega_0} (a_1e^{i\theta},\cdots, a_N e^{i\theta}, d_1,\cdots, d_N).
\end{align*}
Let $k$ be a divisor of $n,$ and suppose that the given configuration of vortices has $k-$fold symmetry. Let $\widehat{\ma} = (a_1, \cdots, a_{N_R})$ be a generating set for this critical point, chosen as in Definition \ref{rem:generating set}. \par Taking into account the rotational invariance of the functional $\calh^{n,\omega_0}$ described above, recall the set $\call_{\wma}$ defined in Assumption A above. Topologically, this set is a circle. Geometrically, $\call_{\wma}$ is a simple closed smooth curve in the \textit{phase space} $(\bD^*)^{N_R}.$ Let us emphasize that $\call_{\wma}$ is \textit{not} the product of circles of radii $|a_1|, \cdots, |a_{N_R}|$ in $\bD,$ with $\bD$ being the \textit{physical space} for the problem \eqref{pvfode}. Finally, to keep notation consistent, points in $\call_{\wma}$ will be denoted with a hat, for instance $\widehat{\mathbf{p}}.$ 
\par 
It then follows that the circle $\call_{\wma}$ is critical for $\calh^{n,\omega_0}$ in the sense that for any $\widehat{\mathbf{b}} \in \call_{\wma},$ the corresponding point $\mathbf{b} \in \bC^N$ is critical for $\calh^{n,\omega_0}.$ It follows that for each $\widehat{\mathbf{b}} \in \call_{\wma},$ the real Hessian $D^2\calh^{n,\omega_0}(\widehat{\mathbf{b}},d)$ identified with a \textit{real} $2N_R \times 2N_R$ matrix has $2N_R$ real eigenvalues and corresponding eigenvectors $(e_1)_{\wmb},\cdots, (e_{2N_R})_{\wmb}$ that are mutually ortho-normal. When there is no confusion, we will drop the subscript $\wmb$ from the eigenvectors in the interest of clarity. 
\par 
Given a point $\widehat{\mathbf{b}} \in \call_{\wma},$ owing to the degeneracy discussed above, cf. Assumption A, the unit tangent vector to $\call_{\wma}$ at $\widehat{\mathbf{b}}$ is a null-eigenvector of $D^2 \calh^{n,\omega_0}.$ Fixing an orientation for the curve $\call_{\wma},$ we may, with no loss of generality, assume that $(e_1)_{\wmb}$ is the unit tangent vector to $\call_{\wma}$ that is consistent with the orientation. For any $\delta > 0,$ we introduce the $\delta-$neighborhood of $\call_{\wma}$ by 
\begin{align*}
\call_{\wma}^\delta := \{ \wmp \in \bR^{2N_R}: \dist(\wmp, \call_{\wma}) < \delta \}.
\end{align*}
We will be interested in sets $\call_{\wma}^\delta$ for suitably small $\delta > 0,$ and we will make this precise shortly. However, before doing so, let us remark that for such $\delta,$ we will continue to denote points of the set $\call_{\wma}^\delta$ \textit{with} hats to remain consistent with \eqref{eq:convention 2}.
\par 
We will be interested in sets $\call_{\wma}^\delta$ where $\delta$ is sufficiently small so that the following hold, see figure 3 below: 
\begin{itemize}
\item The closure $\overline{\call_{\wma}^\delta}$ is contained strictly in $(\bD^*)^{N_R}.$ 
\item The neighborhood $\call_{\wma}^\delta$ admits the following system of local coordinates $(y_1, \cdots, y_{2N_R})$ defined using $(e_1, \cdots, e_{2N_R})$ introduced above: given a point $\wmp \in \overline{\call_{\wma}^\delta},$ for $\delta > 0$ sufficiently small, there exists a \textit{unique} point $\wma_\theta := e^{i\theta}\wma \in \call_{\wma}$ which is closest to $\wmp.$ Set $y_1(\wmp) = \theta.$ Next, denote by $\pi_\theta$ the $2N_R -1 $ dimensional plane passing through $\wma_\theta$ and normal $(e_1)_{\wma_\theta}.$ Then for $\delta$ as above, each point $\widehat{\mathbf{q}} \in \call_{\wma}^\delta$ verifying $\widehat{\mathbf{q}} \in \pi_\theta$ has first component $y_1(\widehat{\mathbf{q}}) = \theta.$ Write $\pi_\theta^\delta := \pi_\theta \cap \call_{\wma}^\delta$ for the $2N_R - 1$ dimensional disc in the plane $\pi_\theta$ of radius $\delta.$ Then naturally, $\pi_\theta^\delta $ has an orthonormal system of coordinates given by $(e_2)_{\wma_\theta}, \cdots, (e_{2N_R - 1})_{\wma_\theta}.$ Furthermore, since for $\delta > 0,$ any point in $\widehat{\mathbf{q}} \in \call_{\wma}^\delta$ belongs to one and only one $\pi_\theta^\delta,$ we will write $\widehat{\mathbf{q}} = (y_1, y_2, \cdots, y_{2N_R})$ where $y_1(\widehat{\mathbf{q}}) = \theta$ is as defined before and the coordinates $(y_2, \cdots, y_{2N_R})$ are the coordinates of $\widehat{\mathbf{q}} \in \pi_\theta^\delta$ in the basis described above. Note that by compactness and smoothness of the curve $\call_{\wma}^\delta,$ this system of coordinates is well-defined for all $\delta$ sufficiently small, say $0 < \delta \leq \delta_*.$ See Figure 3. 
\end{itemize}
\begin{figure} \centering
\includegraphics[scale=.25]{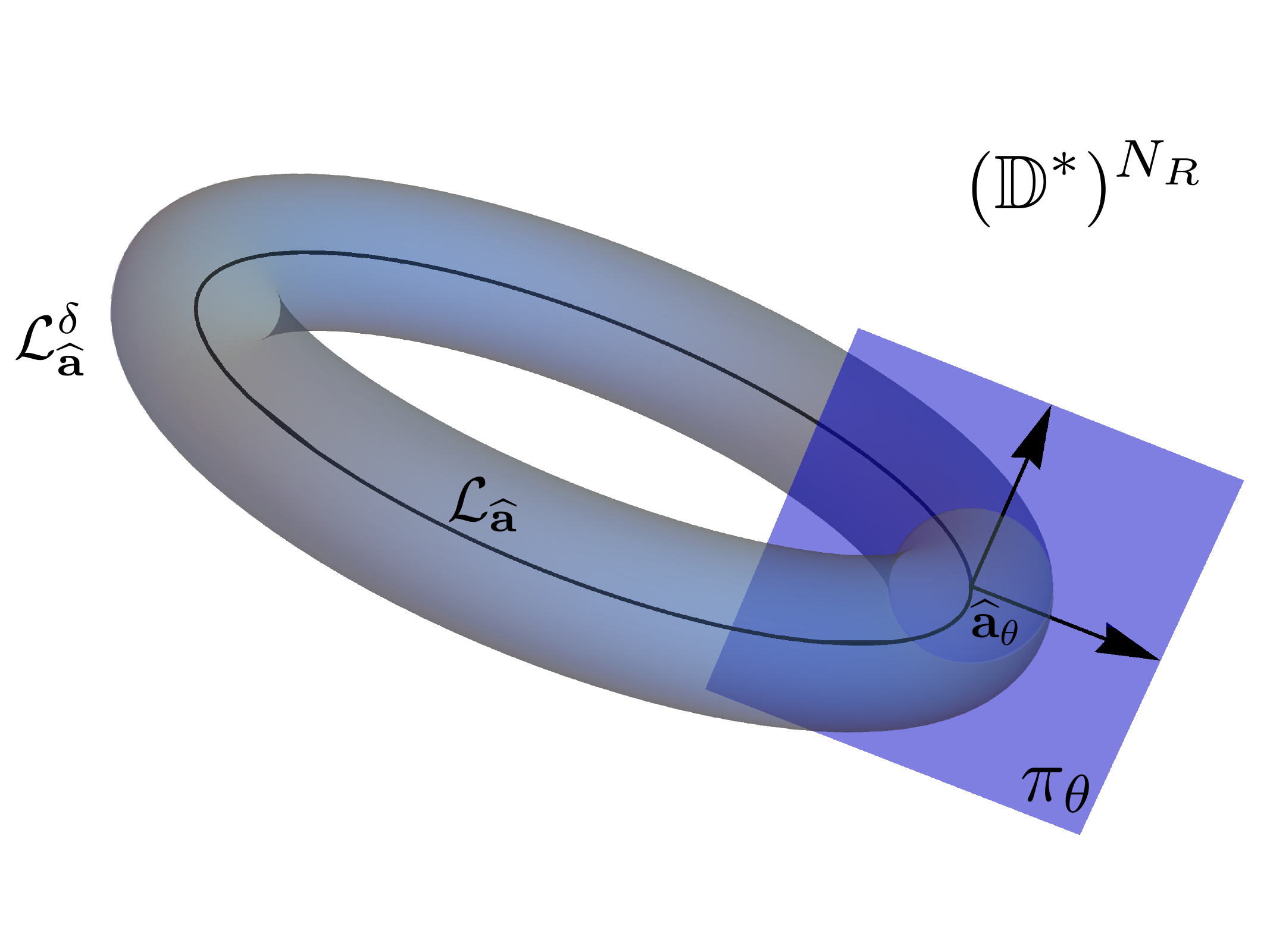}
\caption{A Schematic Picture of a Tubular Neighborhood $\call_{\wma}^\delta \subset (\bD^*)^{N_R}$. The curve along the center line denotes the critical curve $\call_{\wma}$. Also indicated is the orthogonal plane $\pi_\theta,$ at $\wma_\theta \in \call_{\wma}.$ Its intersection $\pi_\theta^\delta$ with $\call_{\wma}^\delta$ is not labeled in the figure. }
\end{figure}
\par  
Let $d$ be fixed as in the statement of the theorem so that we may frequently suppress it from our notations. Given local coordinates as in the preceding paragraph, by Assumption A there exist real numbers $\lambda_1 = 0$ and $\lambda_2, \cdots, \lambda_{2N_R} > 0,$ and a positive integer $1 \leq S \leq 2N_R$ such that for any $\wmp = (y_1,\cdots, y_{2N_R}) \in \call_{\wma}^\delta,$ we may write 
\begin{align}
\label{eq:hessian}
\calh^{n,\omega_0}(\wmp) - \calh^{n,\omega_0}(\wma) = -\sum_{i=1}^S \lambda_i y_i^2 + \sum_{i=S+1}^{2N_R} \lambda_i y_i^2 + o(\dist(\wmp,\call_{\wma})^2).
\end{align}
The case $S=1$ corresponds to a degenerate local minimum of the functional $\calh^{n,\omega_0}$ whereas the case $S=2N_R$ corresponds to a degenerate local maximum. Intermediate values of $S$ correspond to degenerate saddle points of $\calh^{n,\omega_0}$ in $\call_{\wma}^\delta.$ We will deal with the three cases individually. Given the coordinates described above, let $\Pi_1$ denote projection of a point to its first coordinate. \par
\textbf{Step 2:} In this step, we introduce the linking structure that will be used in case $\wma$ is a saddle point, i.e. $2 \leq S \leq 2N_R-1.$ Given $\theta \in [0,2\pi),$ set 
\begin{align*}
B_\theta := \{\wmy \in \call_{\wma}^\delta : \Pi_1(\wmy) = \theta \}. 
\end{align*}
We remark that the sets $B_\theta$ and the set $\pi^\delta_\theta$ from Step 1 above are the same, but we prefer to use the more suggestive notation $B_\theta$ which also suppresses the dependence on $\delta.$ Now, write 
\begin{align*}
D_\theta &:= \{\wmy = (y_1,y_2,\cdots, y_{2N_R}) \in B_\theta: y_{S+1} = y_{S+2} = \cdots = y_{2N_R} = 0\},\\
D_\theta^\perp &:= \{\wmy = (y_1,y_2,\cdots, y_{2N_R}) \in B_\theta: y_2 = \cdots y_S = 0\}.
\end{align*}
Set 
\begin{align} \label{eq:delta1}
\delta_1 := \min \left\{ \min_{\wmy \in \partial D_\theta} \calh^{n,\omega_0} (e^{i\theta}\wma) - \calh^{n,\omega_0}(\wmy), \min_{\wmy \in \partial D_\theta^\perp} \calh^{n,\omega_0}(\wmy) - \calh^{n,\omega_0} (e^{i\theta}\wma) \right\}
\end{align}
By possibly reducing the value of $\delta_* >0,$ we may assume that $\delta_1 > 0.$ \\
\begin{claim} \label{lem:fdlinking}
The sets $\partial D_\theta$ and $D_\theta^\perp$ are homotopically linked in $B_\theta^\delta.$ 
\end{claim}
\begin{proof}[Proof of Claim \ref{lem:fdlinking}]We notice that $\partial D_\theta \cap D_\theta^\perp = \emptyset.$ Let $\Pi_\theta$ denote the projection map from $B_\theta$ onto $D_\theta.$ We must show that for any $h \in C(D_\theta,B_\theta)$ that satisfies $h|_{\partial D_\theta} = id|_{\partial D_\theta}$ one has that $h(D_\theta) \cap D_\theta^\perp \neq \emptyset.$ Note, by the definition of our sets $\partial D_\theta$ and $D_\theta^\perp$, this is equivalent to solving the equation 
\begin{align*}
\Pi_\theta \circ h(d) = 0, \hspace{1cm} d \in D_\theta.
\end{align*}
To achieve this, it suffices to show that 
\begin{align*}
\deg( \Pi_\theta \circ h, \partial D_\theta,0) \mbox{ is well defined and is different from } 0.
\end{align*}
 Define a one parameter family of maps $h_t : \overline{D_\theta} \to B_\theta$ by 
\begin{align*}
h_t = (1-t)id + t h.
\end{align*}
It follows that $h_t|_{\partial D_\theta} = id|_{\partial D_\theta}$ for every $t \in [0,1].$ Consequently, 
\begin{align*}
\deg( \Pi_\theta \circ h_t, \partial D_\theta,0) \mbox{ is well defined }.
\end{align*}
By homotopy invariance of the Brouwer degree, the proof of the Claim follows since $\deg(id,\partial D_\theta, 0) = 1.$ 
\end{proof}
In fact, defining the sets 
\begin{align*}
D &:= \{ \wmp = (y_1, \cdots, y_{2N_R}) \in \call_{\wma}^\delta: y_{S+1} = \cdots = y_{2N_R} = 0, \dist(\wmp , \call_{\wma}) < \delta \}, \\
\partial D &:= \{ \wmp = (y_1, \cdots, y_{2N_R}) \in \call_{\wma}^\delta: y_{S+1} = \cdots = y_{2N_R} = 0, \dist(\wmp , \call_{\wma}) = \delta \},\\
D^\perp &:= \{ \wmp = (y_1,\cdots, y_{2N_R}) \in \call_{\wma}^\delta : y_2 = \cdots = y_S = 0  \},
\end{align*}
then the proof of Claim \ref{lem:fdlinking} shows that $\partial D$ and $D^\perp$ are homotopically linked in $\call_{\wma}^\delta.$ \par 
\textbf{Step 3:} In this step, we construct a gradient flow that will be used to continuously deform the linked sets from Step 2, in a manner that preserves the linking structure. Our desired critical point $v_\e$, as in the statement of the Theorem, will arise as the large time limit of this flow. Before describing the flow, we construct well prepared initial data, and embed the linking structure from the preceding step into the Sobolev space $H^1.$ Given $\wmp \in \call_{\wma}^\delta,$ let $\mathbf{p} \in (\bD^*)^N$ be the associated symmetric configuration of vortices as explained in Definition \ref{rem:generating set}. Then let $w^\e_{\wmp} \in H^1_{g_n}(\bD, \bC)$ denote the map constructed in Proposition \ref{prop:BBH construction} satisfying the following properties: 
\begin{itemize}
\item One has $w^\e_{\wmp}  \in C(\overline{\bD}) \cap H^1_{g_n}(\bD, \bC).$ 
\item One has $|w^\e_{\wmp}(x)| \leq 1$ for all $x \in \bD$ and $w^\e_{\wmp}(x) = 0$ if and only if $x = p_j, j = 1, \cdots , N.$ 
\item The Ginzburg-Landau energy $E_\e (w^\e_{\wmp})$ satisfies 
\begin{align}
\label{eq:log bound init data}
E_\e (w^\e_{\wmp} ) = N \left( \pi \log \frac{1}{\e} + \gamma\right) + W(\mathbf{p},d) + o_\e(1) 
\end{align}
as $\e \to 0,$ where $\gamma \in \bR$ is a universal constant that we recall below.  
\item The momentum $\scrj(w^\e_{\wmp}) $ satisfies 
\begin{align}
\label{eq:momentum const}
\scrj(w^\e_{\wmp}) = \frac{\pi}{m}\scrj_0(\mathbf{p},d) + o_\e(1)
\end{align}  
as $\e \to 0.$ 
\item Finally one has the symmetry condition 
\begin{align} \label{eq:init data symm}
w^\e_{\wmp}\left(r, \theta + \frac{2\pi}{k}\right) = w^\e_{\wmp}(r,\theta). 
\end{align}
\end{itemize} 
We briefly digress to recall the definition and some properties of the universal constant $\gamma$ from \eqref{eq:log bound init data}.
\par
Given $\e > 0$ and $R > 0$ recall the quantity $I(\e,R)$ defined in \cite{BBH}, Chapter III: 
\begin{align*}
I(\e,R) := \min_{u \in H^1_g(B_R(0))} \left\{ \int_{B_R(0)} \frac{|\nabla u|^2}{2} + \frac{(1-|u|^2)^2}{4\e^2}\,dx: g(x) = \frac{x}{|x|}  \right\}.
\end{align*}
Furthermore set $I(t):= I(t,1)$ for $t > 0.$ Then by \cite{BBH} Lemma III.1, the function $t \mapsto I(t) + \pi \log t$ is non-decreasing for $t \in(0,1).$ Finally, recall from \cite{BBH}, Lemma IX.I that the limit 
\begin{align} \label{eq:gamma}
\gamma := \lim_{t \to 0^+} (I(t) + \pi \log t)
\end{align}
exists finitely.  \par 
We now return to the main goal of Step 3, namely construction of a gradient flow. For any $\wmp \in \bD,$ let $u^\e = u^\e(x,t;\wmp)$ denote the unique smooth solution to the initial boundary value problem 
\begin{align}
\label{eq:heat eq}
\frac{\partial u}{\partial t} = \Delta u + \frac{u(1-|u|^2)}{\e^2} &- \omega_0 \left( k u + \frac{1}{m}(y^\perp \cdot \nabla )u^\perp \right), \hspace{.8cm} &x \in \bD, t \in \bR_+,\\ \label{eq:heat init data}
u(x,0;\wmp) &= w^\e_{\wmp}(x), & x \in \bD,\\ \label{eq:heat bc}
u(x,t;\wmp) &= g_n(x), & x \in \bdry, t \in \bR_+.
\end{align}
Let us point out that the system \eqref{eq:heat eq}-\eqref{eq:heat bc} is the gradient flow of $\cale^{\omega_0}.$ Global existence and regularity for this system follows from a parabolic analogue of Lemma \ref{lem:maximum princple} and standard theory for reaction diffusion systems, see for instance \cite{Smoller}, pg. 210. Furthermore, for each fixed time $t \geq 0,$  the map $\wmp \mapsto u_\e(\cdot,t; \wmp) \in H^1_{g_n}(\bD, \bC)$ is continuous, where the target space is given the norm topology. We also have the following Claim which asserts that the heat flow described above preserves symmetry of the initial configuration of vortices. 
\begin{claim} \label{lem:symmetry lemma}
Let $\wmp \in D,$ and $\e > 0$ be fixed. Then the solution $u^\e$ to the system \eqref{eq:heat eq}-\eqref{eq:heat bc} satisfies 
\begin{align}
u^\e\left(r,\theta + \frac{2\pi}{k},t;\wmp\right) = u^\e( r, \theta,t;\wmp),
\end{align}
for every $r \in (0,1], \theta \in [0,2\pi),$ and $t \geq 0.$ 
\end{claim}
\begin{proof}[Proof of Claim \ref{lem:symmetry lemma}] 
We start by noting that a function $U \in H^1 (\bD, \bC)$ which satisfies a symmetry property $U(r, \theta + \frac{2\pi}{k}) = U(r, \theta)$ has a Fourier series development 
\begin{align*}
U(r,\theta) \sim \sum_{j \in \bZ, k |j } c_j (r) e^{ij\theta},
\end{align*}
where we say $k |j$ if $k$ divides $j.$ In other words, Fourier coefficients corresponding to frequencies that are not multiples of $k$ are absent. In the case at hand, we can develop the solution in Fourier series
\begin{align} \label{eq:sol fourier}
u_\e(r,\theta,t;\wmp) = \sum_{j \in \bZ} c_j(r,t) e^{ij\theta},
\end{align}
where we suppress the dependence of the Fourier coefficients on $\e$ and on $\wmp.$ Also, the initial condition $w^\e_{\wmp}$ satisfies symmetry property \eqref{eq:init data symm}. Consequently it admits a Fourier series development 
\begin{align} \label{eq:data fourier}
w^\e_{\wmp}(r,\theta) = \sum_{j \in \bZ, k| j } \mathfrak{c}_j(r)e^{ij\theta}. 
\end{align}
For convenience, set $\mathfrak{c}_j = 0$ if $j$ is not a multiple of $k.$ Since $u_\e$ satisfies the parabolic PDE \eqref{eq:heat eq}, we get the following system of partial differential equations for the Fourier coefficients. 
\begin{align}
\label{eq:PDE Fourier}
\frac{\partial c_j(r,t)}{\partial t} &= \frac{\partial^2 c_j}{\partial r^2} + \frac{1}{r}\frac{\partial c_j}{\partial r} - \frac{j^2}{r^2} c_j + \frac{c_j}{\e^2} - \frac{1}{\e^2}\sum_{p,q,m \in \bZ: p-q + m = j} c_p \overline{c}_q c_m - \omega_0\left(k  - \frac{j}{m}  \right)c_j(r), \hspace{1cm} &0 < r < 1, t \in \bR_+, \\ \notag
c_j(1,t) &= \delta_{jn} & t \in \bR_+, \\ \notag
c_j(r,0) &= \mathfrak{c}_j(r), & 0 < r \leq 1. 
\end{align}
The above system holds for each $j \in \bZ,$ and as usual $\delta_{jn}$ denotes Kronecker's delta symbol. When $j$ is not a multiple of $k,$ one has $0$ boundary and initial conditions for $c_j$; furthermore for such $j$, if $p - q + m = j$ for integers $p,q,m$ then $k$ does not divide at least one of $p,q$ or $m.$ Consequently, by uniqueness $c_j \equiv 0$ is the only solution for $j$ that are not a multiple of $k.$ This completes the proof of Claim \ref{lem:symmetry lemma} and Step 3. 
\end{proof}
\textbf{Step 4:} In this step we deal with the case that $\wma$ is a degenerate local minimum of $\calh^{n,\omega_0},$ i.e. $S=1$ in \eqref{eq:hessian}. Note first that in this case, necessarily the degrees of the vortices are all $+1$ and so $N = n >0:$ this is because, in the presence of a pair vortices of opposite degrees $\pm 1,$ any small change in vortex locations that reduces the distance between this pair of vortices keeping the other vortices fixed, reduces the energy $\calh^{n,\omega_0}.$ 
\par  We consider the gradient flow system \eqref{eq:heat eq}-\eqref{eq:heat bc} with initial data $w^\e(x) := w^\e_{\wma}(x)$ and denote the resulting solution by $u^\e(x,t).$ In particular, note that for the sake of simplicity, throughout Step 4, we do not indicate the dependence of $u^\e(x,t;\wma)$ on $\wma.$ Set 
\begin{align*}
\alpha := \frac{1}{4n +4} \min\{ \delta_*, |a_j - a_i|, i \neq j \in 1, \cdots n \}. 
\end{align*} 
Recall here that $\delta_* > 0$ was introduced in Step 2 above. Set
 \begin{align*}
 \calk_{\ma} := \{e^{i\theta} \ma: 0 \leq \theta < 2\pi\},
 \end{align*} 
and $\calk_{\ma}^\alpha$ to be the $\alpha-$neighborhood of $\calk_{\ma}$ defined similar to $\call_{\wma}^\delta$ from Step 1. We remark that we are using $\mathcal{K}$ for neighborhood sets in $({\bD}^*)^N$ and $\mathcal{L}$ for neighborhood sets in $(\bD^*)^{N_R}.$ The main claim of this step is: \par 
\begin{claim} \label{claim:close}
Define the set 
\begin{align*}
G_\e(t) := \left\{x \in \bD : |u_\e(x,t)| \leq \frac{1}{2} \right\}.
\end{align*}
Then for any $\e >0$ sufficiently small and any $t \geq 0,$ if $x_\e \in G_\e(t),$ then 
\begin{align} \label{eq:close to where we started from}
\min_{j=1,\cdots,n} \big[\min_{\theta \in [0,2\pi)} \dist\left(x_\e, a_j e^{i\theta}\right)\big] < \alpha. 
\end{align}
\end{claim}
\begin{proof}[Sketch of the Proof of Claim \ref{claim:close}]
The proof of this Claim follows very closely the proof of Theorem A of \cite{LinAIHP}. We recall the main ideas for the convenience of the reader. Suppose by way of contradiction, there exists a first positive time $T_\e \in (0,\infty)$ such that \eqref{eq:close to where we started from} fails. 
In other words, there exists a point $x_1^\e \in \bD,$ a critical point $\mathbf{a} = (a_1,\cdots, a_n) \in \calk_{\ma},$ and points $(x_j^\e)_{j=2}^n \in \bD$ and  such that, without loss of generality,
\begin{align} \label{eq:contradiction assumption}
|x_1^\e - a_1| = \alpha,& \hspace{1cm} |u_\e(x_1^\e,T_\e)| \leq \frac{1}{2},& \\ \notag
|x_j^\e - a_j| \leq \alpha,& \hspace{1cm} |u_\e(x_j^\e,T_\e)| \leq \frac{1}{2}, \hspace{.3cm} &j=2,3,\cdots,n.
\end{align}
Indeed, the point $x_1^\e$ and $a_1$ such that $\ma = (a_1,\cdots, a_n) \in \calk_\ma$ are obtained from the contradiction hypothesis, and subsequently, we choose the points $(x_j^\e)_{j=2}^n$ to satisfy the above conditions. 
Note that by virtue of the symmetry Claim \ref{lem:symmetry lemma}, we may assume that the points $(x_j^\e)$ satisfy the $k-$fold symmetry property. \par  
Set $\mu := \frac{1}{16n + 1},$ and $V_\e(x) := u_\e(x,T_\e).$ We will also find it convenient to use the notation 
\begin{align*}
\cale^{\omega_0}(w,A) := \int_A \frac{1}{2} |\nabla w|^2 + \frac{(1-|w|^2)^2}{4\e^2} \,dy + \frac{\omega_0}{2} \int_A k |w|^2 + \frac{1}{m} w \cdot (y^\perp \cdot \nabla )w^\perp \,dy,
\end{align*}
for all $w \in H^1_{g_n}(\bD)$ having the $k-$fold symmetry property, where $A \subset \bD$ is any Lebesgue measurable set. We will also use similar notation for the functionals $E_\e$ and $\scrj.$ 
\par 
By Fubini's theorem, the energy decreasing property of the gradient flow, and the $\log \frac{1}{\e}$ upper bound on the initial data, it can be easily seen that, for each $j = 1,\cdots, n,$ there exists $\mu_j \in [\mu, 2\mu]$ and a positive number $C(\mu)$ satisfying the estimate 
\begin{align}
\label{eq:Struwe estimate}
\e^{\mu_j} \cale^{\omega_0}\left(V_\e,\partial B(x_j^\e, \e^{\mu_j})\right) \leq C(\mu) 
\end{align}
for all $\e >0$ sufficiently small. We remark that since $V_\e$ is actually smooth, its restriction to the one dimensional set as in \eqref{eq:Struwe estimate} is actually meaningful. Since \eqref{eq:Struwe estimate} holds for all $\e >0$ small enough, it follows that $d_j := \deg(V_\e, \partial B(x_j^\e, \e^{\mu_j}),0)$ is well defined. By the Structure Theorem (cf. \cite{LinCPAM}, Theorem 2.4), we may assume that the points $(x_j^\e)_{j=1}^n$ satisfying \eqref{eq:contradiction assumption}, additionally satisfy 
\begin{align} \label{eq:bigger than half}
\left|V_\e|_{\partial B(x_j^\e,\e^{\mu_j})}\right| &\geq \frac{1}{2}, \\ \label{eq:degree condition}
d_j &= +1 \hspace{1cm} \mbox{ for all } \e >0 \mbox{ sufficiently small. }
\end{align}  
for all $j = 1,\cdots, n$. Set 
\begin{align*}
\bD_\e := \bD \backslash \bigcup_{j=1}^n B(x_j^\e,\e^{\mu_j}). 
\end{align*}
Next set $U_\e$ to be the solution to the variational problem 
\begin{align} \label{eq:variational problem punctured domain}
\inf \cale^{\omega_0}(w, \bD_\e)
\end{align}
among functions $w \in H^1(\bD_\e)$ with $w|_{\partial \mathbb{D}_\e} = V_\e$ and $w$ having the $k-$fold symmetry property. Existence of minimizers $U_\e$ in this admissible set follows by an easy application of the direct method. Since the Brouwer degree of $V_\e|_{\partial \bD_\e}$ is zero, the estimate \eqref{eq:bigger than half} along with arguments as in \cite{LinAIHP} Pg. 615, yield that  
\begin{align*}
|U_\e| \geq \frac{1}{2} \mbox{ on } \bD_\e.
\end{align*}
Finally, by compactness we may assume (upon perhaps passing to a subsequence) that $x_j^\e \to \overline{a}_j$ as $\e \to 0$ for each $j = 1,\cdots n,$ with 
\begin{align} \label{eq:contradiction assumption limit}
|\overline{a}_1 - a_1 | &= \alpha, \\
|\overline{a}_j - a_j| &\leq \alpha, \hspace{1cm} \mbox{ for } j = 2,3, \cdots, n. 
\end{align}
Define the function
\begin{align*}
\overline{V}_\e (y) := \left\{ 
\begin{array}{cc}
V_\e(y) \hspace{1cm} & y \in \bigcup_{j=1}^n B(x_j^\e,\e^{\mu_j}), \\
U_\e(y) \hspace{1cm} & y \in \bD_\e.
\end{array}
\right.
\end{align*}
We will use $\overline{V}_\e|_{\bD_\e}$ as a competitor to the variational problem \eqref{eq:variational problem punctured domain}.
Then, on the one hand by the energy decreasing property of the gradient flow \eqref{eq:heat eq}-\eqref{eq:heat bc} , 
\begin{align} \notag
\cale^{\omega_0} (\overline{V}_\e, \bD) \leq \cale^{\omega_0}(V_\e,\bD) &\leq \cale^{\omega_0}(u_\e(\cdot,0))\\ \notag
&= \cale^{\omega_0}(w_\e) \\ \label{eq:upperboundstep5 claim}
&= n\left( \pi \log \frac{1}{\e} + \gamma \right) + \pi \calh^{n,\omega_0}(\mathbf{a}) + o_\e(1).
\end{align}
On the other hand, from \eqref{eq:degree condition} we are led to conclude following \cite{LinAIHP} and the continuity under weak topology in $H^1$ of the momentum functional $\scrj$ that on compact subsets of $\bD_\e,$ as $\e \to 0$ 
\begin{align} \label{eq:CHM convergence}
U_\e \to U^*_{\overline{\ma}}
\end{align}
where $U^*_{\overline{\ma}}$ is the canonical harmonic map associated to the points $\overline{\ma}.$ In fact we will sketch a different proof of this during the course of the proof of Theorem \ref{thm:asymptotics}. Consequently, by \cite{BBH} Lemma VIII.2, for any sufficiently small \textit{fixed} $\rho \ll \alpha,$ using the notation $\bD_\rho(\overline{\mathbf{a}}) := \bD \backslash \bigcup_{j=1}^n B(\overline{a_j},\rho),$ we find
\begin{align*}
\cale^{\omega_0}\left(U_\e,\bD_\rho(\overline{\mathbf{a}})\right) = n \pi \log \frac{1}{\rho} + \pi\calh^{n,\omega_0} (\overline{\ma}) + o_\e(1) + O(\rho). 
\end{align*}
We will make our choice of $\rho$ (\textit{independently} of $\e$ ) at the end of the proof. We can hence conclude that 
\begin{align} \notag
\cale^{\omega_0}(\overline{V}_\e, \bD) &= E_\e\left(\overline{V}_\e, \bD_\rho(\overline{\ma})\right) + E_\e\left(\overline{V}_\e,\bigcup_{j=1}^n B(\overline{a}_j,\rho)\right) - \omega_0 \scrj(\overline{V}_\e, \bD)\\  \notag
& = E_\e\left(U_\e,\bD_\rho(\overline{\ma}) \right) + E_\e\left(\overline{V}_\e, \bigcup_{j=1}^n B(\overline{a}_j,\rho) \right) - \omega_0\scrj\left(\overline{V}_\e, \bD \right)\\ \notag
& \geq n \pi \log \frac{1}{\rho} + \pi\calh^{n,\omega_0} (\overline{\ma}) + O(\rho) + o_\e(1) + n I(\e, \rho), \\ \label{eq:lowerboundforstep 5 claim}
&\geq n \left( \pi \log \frac{1}{\e} + \gamma\right) + \pi\calh^{n,\omega_0} (\overline{\ma}) + O(\rho) + o_\e(1)
\end{align}
where, we have used \eqref{eq:CHM convergence} and the remarks at the start of the section on the quantity $I(\e,\rho).$ Finally, we can put together the upper and lower bounds from \eqref{eq:upperboundstep5 claim}-\eqref{eq:lowerboundforstep 5 claim} to arrive at
\begin{align*}
\pi \calh^{n,\omega_0}(\overline{\ma}) + O(\rho) - \pi \calh^{n,\omega_0}(\ma) \leq o_\e (1),
\end{align*}
a contradiction to \eqref{eq:hessian}, by choosing $\rho$ sufficiently small that $O(\rho)$ terms are controlled by $\min(\lambda_2,\cdots, \lambda_n)\times \frac{\delta_*^2}{50 n} > 0$ since $\overline{\ma} \not\in \calk_\ma$ by virtue of \eqref{eq:contradiction assumption limit}.  
\end{proof}
Having proven Claim \ref{claim:close}, we can now conclude that for each time $m = 1,2,\cdots,$ the map $u_\e(\cdot,m)$ has essential zeroes $b^{(m)}_1,\cdots, b^{(m)}_n \in \bD$ that are well-defined up to errors that are at most $4{\e}^{\alpha_0}$ (\cite{LinMinMax}, Theorem 1.1 ). By virtue of the symmetry Claim \ref{lem:symmetry lemma}, we may take $b^{(m)}_1,\cdots, b^{(m)}_n$  having the same symmetry property as $\ma.$ Finally using Claim \ref{claim:close} we conclude that 
\begin{align*}
\mathbf{b}^{(m)} := (b^{(m)}_1,\cdots, b^{(m)}_n) \in \calk_{\ma}^\alpha \subsetneq \calk_{\ma}^\delta.
\end{align*}
Using arguments from the proof of the Claim \ref{claim:close}, we can now prove that 
\begin{align} \notag
\cale^{\omega_0}(u_\e(\cdot,m)) & \geq n \left( \pi \log \frac{1}{\e} + \gamma \right) + \pi\calh^{n,\omega_0} (\mathbf{b}^{(m)}) + o_\e(1)  \\ \label{eq:lowerboundmincase}
& \geq n\left( \pi \log \frac{1}{\e} + \gamma\right) + \pi\calh^{n,\omega_0} (\mathbf{a}) + o_\e(1).
\end{align}
By compactness we may assume that as $m \to \infty,$ we have $\mathbf{b}^{(m)} \to \mathbf{p}_*$ with the $k-$fold symmetry property. Moreover by Claim \ref{claim:close} above, $\mathbf{p}_* \in \calk_{\mathbf{a}}^\alpha \subsetneq \calk_{\mathbf{a}}^{\delta_*}. $ 
\par 
Now, on the one hand, by the energy decreasing property of gradient flow, we find for each $t > 0,$
\begin{align} \label{eq:upperboundmincase}
\cale^{\omega_0} (u_\e(\cdot,t)) \leq \cale^{\omega_0} (u_\e(x,0)) \leq n\left(\pi \log \frac{1}{\e} +  \gamma \right) + \pi\calh^{n,\omega_0}(\ma) + o_\e(1).
\end{align}
Using the gradient flow property and Theorem 2 of \cite{LeonSimon} by L. Simon, we conclude that $v_\e(x) := \lim_{t \to \infty} u_\e(x,t)$ exists for each fixed $\e > 0$ in $H^1,$ and is a critical point of $\cale^{\omega_0}.$ Furthermore, $v_\e$ is a solution to \eqref{ansatzpde}-\eqref{ansatzbc} with $\omega = \omega_0.$    
On the other hand though, for each fixed $t,$ using the energy decreasing property and using \eqref{eq:lowerboundmincase}, 
\begin{align*}
\cale^{\omega_0} (u_\e(\cdot,t)) \geq \limsup_{m \to \infty} \cale^{\omega_0} (u_\e(\cdot,m)) \geq  n \left( \pi \log \frac{1}{\e} + \gamma \right) + \pi\calh^{n,\omega_0} (\mathbf{a}) + o_\e(1). 
\end{align*}
Combining the last two estimates we find that  
\begin{align*}
\left| \cale^{\omega_0} (u_\e(\cdot,t)) - n \left( \pi \log \frac{1}{\e} + \gamma \right) - \pi\calh^{n,\omega_0} (\ma) \right| = o_\e(1),
\end{align*}
\textit{uniformly} in $t$ as $\e \to 0^+.$ 
Passing to the limit $t \to \infty,$ we obtain 
\begin{align}
\left| \cale^{\omega_0} (v_\e) - n \left( \pi \log \frac{1}{\e} + \gamma \right) - \pi\calh^{n,\omega_0} (\ma) \right| = o_\e(1). 
\end{align}
Plugging this back into the ansatz \eqref{ansatz} completes the proof of the theorem in the case $S = 1,$ i.e. the local minimizer case. 
\par
In Steps 6 through 8 we will complete the proof in the case $2 \leq S \leq 2N_R-1,$ i.e. the saddle point case. At that point, with our experience from the local minimizer and saddle point cases, the local maximizer case $S = 2N_R$ will be an easy modification of these earlier cases, and we will briefly pursue it at the end of Step 8.  
\par 
\textbf{Step 5:} We next use the gradient flow defined in equations \eqref{eq:heat eq}-\eqref{eq:heat bc} to deform the set $D$ defined at the end of Step 2 in such a way that the deformed one parameter family of sets stays linked with $D^\perp.$ Let $\wmp \in D$ be arbitrary. Note that for each $t \geq 0, u_\e(\cdot,t; \wmp) \in H^1_{g_n}(\bD, \bC),$ and by the energy decreasing property of the gradient flow \eqref{eq:heat eq}, 
\begin{align} \notag
\cale^{\omega_0}(u_\e(\cdot,t;\wmp)) &\leq \cale^{\omega_0}(u_\e(\cdot,0;\wmp)) \\ \notag
& = \cale^{\omega_0} (w^\e_{\wmp})\\ \notag
&= E_\e(w^\e_{\wmp}) + \omega_0 \scrj(w^\e_{\wmp}) \\ \notag
&= N\left( \pi \log \frac{1}{\e} + \gamma \right) + \pi \calh^{n,\omega_0}(\wmp,d) + o_\e(1) \\ \label{eq:init data upper bound}
&\leq N\left( \pi \log \frac{1}{\e} + \gamma \right) + \pi \calh^{n,\omega_0}(\wma,d) + o_\e(1) .
\end{align}
where in the last line we have used the fact that $\wmp \in D.$ 
\par We will use this observation along with Proposition \ref{lem: proj lemma} to construct a one parameter family of deformations $h_t(\wmp) = h(t,\wmp)$ of $\overline{D}.$ 
Since $\overline{\call_{\wma}^\delta}$ is a compact subset of $(\bD^*)^{N_R},$ our choice of initial data (cf. Equations \eqref{eq:log bound init data} and \eqref{eq:momentum const}) and estimate \eqref{eq:init data upper bound} imply that $\{ u_\e(\cdot, t;\wmp): \wmp \in \overline{D}\}$ satisfy the hypothesis of Proposition \ref{lem: proj lemma} for each point $\wmp \in D,$ for at least small $t > 0$. Let $\mathcal{P}$ denote the projection map that results from Proposition \ref{lem: proj lemma}. For any $\wmp \in \overline{D},$ let $\tau_{\wmp} \geq 0$ to be the first time such that 
\begin{align}
\label{eq:time of exit}
\mathcal{P}\left(u_\e(\cdot,\tau_{\wmp};\wmp)\right) \in \partial \call_{\wma}^\delta. 
\end{align}
If $\mathcal{P}(u_\e(\cdot,t,\wmp)) \not\in  \partial \call_{\wma}^\delta$ for \textit{any} $t > 0,$ then set $\tau_{\wmp} = +\infty.$ Given this convention, define
\begin{align}
h(t,\wmp) := \left\{
\begin{array}{cc}
\wmp \hspace{1cm}, & t = 0, \\
\mathcal{P} \left(u_\e(\cdot,t;\wmp)\right), \hspace{1cm} & 0 \leq t \leq \tau_{\wmp}, \\ 
\mathcal{P} \left(u_\e(\cdot,\tau_{\wmp};\wmp)\right) \hspace{1cm} & t \geq \tau_{\wmp}.
\end{array}
\right.
\end{align}
Then it is clear that $h(t,\wmp) : \bR_+ \times D \to \overline{\call_{\wma}^\delta}$ is continuous. Furthermore, if $h(t,\wmp)$ lies in the interior of the set $\call_{\wma}^\delta$ then by Prop. \ref{lem: proj lemma}, $h(t,\wmp)$ is at most $4\e^{\alpha_0}$ away from the essential zeroes of $u_\e(\cdot,t;\wmp).$ 
\par 
\textbf{Step 6:} We claim that the sets $\{h(t,\wmq): \wmq \in \partial D\}$ and $D^\perp$ are homotopically linked for all $t \geq 0.$ In other words, we must show that for all time $t \geq 0,$ 
\begin{align*}
\{h(t,\wmq): \wmq \in \partial D \} \cap D^\perp = \emptyset
\end{align*}
implies that 
\begin{align*}
\{h(t,\wmq): \wmq \in D \} \cap D^\perp \neq \emptyset. 
\end{align*}
To prove this, first observe that $h(0,\wmq) = \wmq,$ and Claim \ref{lem:fdlinking} proves the claim for $t = 0.$ To prove the claim for $t > 0,$ let $\sigma > 0$ be small, such that 
\begin{align}
\mbox{ If } \hspace{.3cm} \wmy \in \call_{\wma}^\delta, \dist(\wmy,\partial D^\perp) < 2\sigma, \hspace{.6cm} \mbox{ then } \hspace{.3cm} \pi\calh^{n,\omega_0}(\wmy) \geq \pi\calh^{n,\omega_0}(\wma) + \frac{2\delta_1}{3}.
\end{align}
\textit{Sub-claim $\mathfrak{C}_1$:} For $t >0, $ if $\wmq \in \partial D, $ then $\dist(h(t,\wmq),\call_{\wma}) > \delta - \frac{\sigma}{4}.$\\
\textit{Sub-claim $\mathfrak{C}_2$:} For $t > 0,$ if $\wmq \in \partial D,$ then $\dist(h(t,\wmq),\partial D^\perp) \geq 2\sigma > 0. $ \\
Granting these two sub-claims for now, they would imply that $\{h(t,\wmq):\wmq \in \partial D\} \cap D^\perp = \emptyset$ for all $t \geq 0.$ Furthermore, $h(0,\wmp) = \wmp$ for all $\wmp \in \overline{D}$ by definition; letting $\Pi:\overline{\call_{\wma}^\delta} \to D$ denote the projection map, we conclude
\begin{align*}
\deg(\Pi \circ h(t,\partial D),\partial D, 0) \mbox{ is well-defined }.
\end{align*}
Consequently, by the homotopy invariance of Brouwer degree,
\begin{align*}
\deg( \Pi \circ h(0,\partial D),\partial D,0) = \deg(\Pi \circ id|_{\partial D},\partial D,0) = \deg(id,\partial D,0)= 1.
\end{align*}
This proves that $\{h(t,\wmq): \wmq \in \partial D\}$ and $D^\perp$ are homotopically linked.\par 
It remains to prove Sub-claims $\mathfrak{C}_1$ and $\mathfrak{C}_2$. We will prove Sub-claim $\mathfrak{C}_1$; the proof of Sub-claim $\mathfrak{C}_2$ is very similar. \\
\textit{Proof of Sub-claim $\mathfrak{C}_1$:} Suppose to the contrary that there exists $\wmp \in \partial D$ such that $\dist(h(t,\wmp),\call_{\wma}) \leq \delta - \frac{\sigma}{4}$ for some $t > 0.$  Then up to errors that are at most $4\e^{\alpha_0},$ we have $h(t,\wmp)$ are the essential zeroes of $u_\e(\cdot,t;\wmp).$ Then following arguments similar to those in the proof of Claim \ref{claim:close} there exists $\rho_1$ independent of $\e,$ such that for $0 < \rho < \rho_1,$ one has 
\begin{align*}
\cale^{\omega_0} \left(u_\e(\cdot,t;\wmp)\right) &\geq N \pi \log \frac{1}{\rho} + N \pi I(\e, \rho) + \pi\calh^{n,\omega_0}(h(t,\wmp)) + O(\rho) + o_\e(1)\\
&\geq N \pi \log \frac{1}{\e} + N \gamma + \pi\calh^{n,\omega_0}(h(t,\wmp)) + O(\rho) + o_\e(1). 
\end{align*}
In the second inequality we have used the definition and properties of the constant $\gamma$ from \eqref{eq:gamma}, Pg.24. By comparing this estimate to \eqref{eq:init data upper bound}, we get a contradiction to \eqref{eq:hessian} upon choosing $\rho \ll \delta - \frac{\sigma}{4}.$ This completes the proof of Sub-claim $\mathfrak{C}_1.$ The proof of Sub-claim $\mathfrak{C}_2$ is identical, and this completes Step 6. 
\par 
\textbf{Step 7:} In this step we prove an $\inf-\sup$ characterization of the critical point. Define now 
\begin{align}
\label{eq:crit value}
c_\e := \inf_{t > 0 } \sup_{\wmp \in \overline{D}} \cale^{\omega_0} \big(u_\e(\cdot,t;\wmp)\big)
\end{align}
\textbf{Claim:}  $c_\e$ is a critical value of $\cale^{\omega_0}.$ \\ 
\textit{Proof of claim:} Suppose to the contrary that $c_\e$ is a regular value of $\cale^{\omega_0}.$ Then by Lemma \ref{lem:Palais-Smale}, there exist positive numbers $\eta_1, \eta_2 > 0,$ such that if
\begin{align}
|\cale^{\omega_0} (u) - c_\e| < \eta_1
\end{align}
then 
\begin{align}
|\nabla \cale^{\omega_0}(u)| > \eta_2. 
\end{align}
By definition of infimum, there exists $t_0 > 0,$ such that 
\begin{align} \label{eq:PS proof}
c_\e \leq \sup_{\wmp \in \overline{D}} \cale^{\omega_0}\big(u_\e(\cdot,t_0;\wmp)\big) \leq c_\e + \eta_1.
\end{align}
But by the energy decreasing property of the gradient flow \eqref{eq:heat eq} and the definition \eqref{eq:crit value}, it follows that \eqref{eq:PS proof} holds for all $t \geq t_0.$ But this means that 
\begin{align*}
\frac{\,d}{\,dt} \sup_{\wmp \in \overline{D}} \cale^{\omega_0} (u_\e(\cdot,t;\wmp)) = - |\nabla \cale^{\omega_0}(u)|^2 \leq - \eta_2^2, \hspace{1cm} \mbox{ for all } t \geq t_0.
\end{align*}
This immediately yields a contradiction to the definition of $c_\e$, cf. \eqref{eq:crit value}, for $t \geq t_0 + \frac{1}{\eta_2^2}.$ This completes the proof of the Claim. 
\par
\textbf{Step 8:} Finally, we conclude the proof of our existence Theorem \ref{thm:existence} following \cite{LinMinMax}. From Step 6 it follows that for each time $m = 1,2, \cdots, $ there exists $\wmp_m \in D$ such that $h(m,\wmp_m) \in D^\perp.$ By arguments similar to those in the proof of Claim \ref{claim:close} we also conclude 
\begin{align*}
\cale^{\omega_0} \left(u_\e(\cdot,m; \wmp_m)\right) &\geq N \left( \pi \log \frac{1}{\e} + \gamma \right) + \pi\calh^{n,\omega_0}(h(m,\wmp_m)) + o_\e(1), \\
&\geq N \left( \pi \log \frac{1}{\e} + \gamma \right) + \pi\calh^{n,\omega_0}(\wma) + o_\e(1)
\end{align*}
where in the second inequality we have used the fact that $h(m,\wmp_m) \in D^\perp.$ Without loss of generality, we may assume that (possibly upon passing to a subsequence) $\wmp_m \to \wmp_*$. Using the energy decreasing nature of \eqref{eq:heat eq} it follows that for any fixed $t >0,$
\begin{align*}
\cale^{\omega_0} (u_\e(\cdot,t;\wmp_*)) = \lim_{m \to \infty} \cale^{\omega_0}(u_\e(\cdot,t;\wmp_m)) \geq \lim_{m \to \infty} \cale^{\omega_0}(u_\e(\cdot,m;\wmp_m)).
\end{align*}
The last two estimates imply that 
\begin{align*}
\left|\cale^{\omega_0} (u_\e(\cdot,t;\wmp_*)) - N\left( \pi \log \frac{1}{\e} + \gamma \right) - \pi\calh^{n, \omega_0}(\wma)\right| = o_\e(1)
\end{align*}
as $\e \to 0$ uniformly in $t.$ Since the dependence of $\cale^{\omega_0}$ on $u, \nabla u$ is analytic, we may use the gradient flow structure of \eqref{eq:heat eq} and \cite{LeonSimon}, Theorem 2 to conclude that the limit 
\begin{align}
v_\e(\cdot) := \lim_{t \to \infty} u_\e(\cdot,t;\wmp_*)
\end{align}
exists in $H^1,$ and is a critical point of $\cale^{\omega_0},$ solving the PDE \eqref{ansatzpde}-\eqref{ansatzbc}. This completes the proof of the fact that there exists a critical point $v_\e$ of $\cale^{\omega_0}$ having the $k-$ fold symmetry property, as in \eqref{eq:init data symm}. Furthermore $v_\e$ satisfies 
\begin{align}
\label{eq:crit point estimate}
\left| \cale^{\omega_0}(v_\e) - N \left( \pi \log \frac{1}{\e} + \gamma\right) - \pi\calh^{n,\omega_0}(\wma)\right| = o_\e(1)
\end{align}
as $\e \to 0.$
\par 
The case $S = 2N_R$ remains, and is now fairly easy. Define the one parameter family of deformations $\{h(t,\wmp): \wmp \in D\}$ as in the saddle case. Then similar to the saddle case, for each $m = 1,2,\cdots,$ it can be shown by easy degree theory arguments as before that there exist $\wmp_m \in D$ such that $h(m, \wmp_m) = \wma.$ This sequence of points gives us the necessary estimates and we reach similar conclusions as in the local minimizer and the saddle cases. 
\par 
In summary, for critical points $\ma$ of the functional $\calh^{n,\omega_0}$ as are under consideration, for $\e $ sufficiently small we have obtained critical points $v_\e$ of $\cale^{\omega_0}$ with the same symmetry property as $\ma$ satisfying the estimate 
\begin{align}
\label{eq:grandestimate}
\left| \cale^{\omega_0} (v_\e) - N\left( \pi \log \frac{1}{\e} + \gamma \right) - \pi\calh^{n,\omega_0}(\ma, d)\right| = o_\e(1)
\end{align}
as $\e \to 0.$ 
\par 
We are (finally!) ready to plug back the ansatz \eqref{ansatz} to obtain a time periodic solution of the Gross-Pitaevskii equation \eqref{eq:GP} whose period is $T = \frac{2\pi m}{\omega_0}.$ This completes the proof of the  theorem concerning existence of a time periodic solution to \eqref{eq:GP} also possessing the symmetry property \eqref{eq:symmetry prop of solution}. 
\end{proof}

Our second main theorem of this section concerns the $\e \to 0$ asymptotics of the periodic solutions constructed in Theorem \ref{thm:existence}.
\begin{theorem} \label{thm:asymptotics} Let $(\ma(t),d)$ be as in the previous theorem, and let $u_\e(x,t)$ be the corresponding time periodic solution constructed there. Then there exists $\theta_* \in [0,2\pi)$ such that 
\begin{align} \label{eq:periodic sols follow ODE}
Ju_\e(\cdot , t) \rightharpoonup \pi \sum_{i=1}^N d_i \delta_{e^{i\theta_*}a_i(t)}, 
\end{align}
for each fixed $t \in \bR,$ where the convergence is in the weak-$*$ topology of $(C^{0,\alpha}_c(\bD))^*$ for every $\alpha \in (0,1).$
\end{theorem}
In fact, our proof will show something in addition to \eqref{eq:periodic sols follow ODE}: let $v_\e$ be related to $u_\e$ as in the statement of the previous theorem. Our proof will show that there exists a finite set $F \subset \bD$ with exactly $N$ points such that as $\e \to 0^+,$ we have, for each $\kappa \in \bN$ and each compact set $K \subset \bD \backslash F$  the convergence $v_\e \to v_*$ holds in $C^\kappa(K).$ Moreover, $|v_*| = 1$ on $\bD \backslash F$ and in fact, $v_*$ is a smooth harmonic map on $\bD \backslash F$ in the sense that 
\begin{align*}
\Delta v_* + |Dv_*|^2 v_* &= 0, \hspace{1cm} \bD \backslash F,\\
v_*|_{\partial \bD} &= e^{in\theta}.
\end{align*}
It is natural to wonder why the angular speed $\omega_0$ which appeared on the RHS of the PDE's for $v_\e$ seemingly disappears in the limiting PDE for $v_*.$ We show an analogue of the so-called \textit{vanishing gradient property} from \cite{BBH} in which we recover the angular speed $\omega_0.$ Indeed, it appears as a sort of Lagrange multiplier for the renormalized energy, thereby deciding vortex location.  
\par The proof of Theorem \ref{thm:asymptotics} requires the following Pohazaev identity.  
\begin{lemma} \label{lem:Pohazaev}
Let $v_\e$ be as in Theorem \ref{thm:existence}. Then 
\begin{align}
\label{eq:Pohazaev}
\frac{1}{2}\int_{\partial \bD} \left| \frac{\partial v_\e}{\partial \nu}\right|^2 - \pi n^2 + \int_{\bD} \frac{(1-|v_\e|^2)^2}{2\e^2} = \omega_0 k \left( \pi - \int_{\bD} |v_\e|^2 \,dy\right) - \frac{\omega_0}{m} \int_{\bD} |y|^2 (Jv_\e) \,dy. 
\end{align}
\end{lemma}

\begin{proof}[Proof of Lemma \ref{lem:Pohazaev}]
Recall that the function $v_\e$ satisfies the Elliptic PDE \eqref{ansatzpde}-\eqref{ansatzbc}. Then, the proof of this Pohazaev identity proceeds as usual by taking the dot product on $(y \cdot \nabla v_\e)$ with each side of the PDE \eqref{ansatzpde}. We suppress the dependence of $v$ on $\e$ for the sake of clarity. We obtain,
\begin{align*}
(y \cdot \nabla v) \cdot \left(\Delta v + \frac{v(1-|v|^2)}{\e^2}\right) = \omega_0 k (y \cdot \nabla v) \cdot v + \frac{\omega_0}{m} (y \cdot \nabla v)\cdot (y^\perp \cdot \nabla v^\perp).
\end{align*}
We integrate on $\bD$ and proceed as usual. Since the calculations for the terms on the left hand side, and the first term on the right hand side are fairly standard, we only compute that 
\begin{align*}
(y \cdot \nabla v)\cdot (y^\perp \cdot \nabla v^\perp) &= -|y|^2 \left( \partial_1v^1 \partial_2v^2 - \partial_1 v^2 \partial_2 v^1\right)\\
&= - |y|^2 (Jv).
\end{align*}
where $\partial_i = \frac{\partial}{\partial y_i}.$ This completes the proof of Pohazaev's identity. 
\end{proof}

\begin{proof}[Proof of Theorem \ref{thm:asymptotics}]
Thanks to the energy estimate \eqref{eq:energy close}, we can apply the Compactness Theorem, cf. \cite[Theorem 3.1]{JerrardSoner}. This says that $Jv_\e \rightharpoonup J$ where $J = \pi \sum_{i=1}^\mathcal{N} k_i \delta_{\alpha_i},$ in the dual Sobolev space $W^{-1,1}(\bD)$. Here $\alpha_i \in \bD, k_i \in \bZ$ and $\sum_{i=1}^\mathcal{N}|k_i|$ is bounded independently of $\e.$ The exact details of $\alpha_i, k_i, \mathcal{N}$ do not matter for the time being since we are only interested in obtaining a uniform estimate on the potential term. Then writing 
\begin{align*}
\int_\bD |y|^2 (Jv_\e) \,dy = \int_\bD (|y|^2 - 1) (Jv_\e)\,dy + \int_\bD Jv_\e \,dy, 
\end{align*}
since $(|y|^2 -1 )|_{\partial \bD} = 0,$ we find $\int_{\bD} (|y|^2 - 1)(Jv_\e)\,dy \to \pi\sum_{i=1}^\mathcal{N} k_i (|\alpha_i|^2 - 1).$ Finally, by Stokes' theorem, 
\begin{align*}
\int_{\bD} Jv_\e\,dy &= \frac{1}{2}\int_{\bD} \nabla \times j(v_\e) 
= \frac{1}{2} \int_{\partial \bD} j(v_\e) \cdot \,d\theta
= \frac{1}{2} \int_{\partial \bD} (i v_\e, (v_\e)_\theta) 
= \frac{1}{2} \int_{\partial \bD} n \,d \theta = n \pi. 
\end{align*}
Putting these together, we obtain that 
\begin{align}
\label{eq:potential estimate}
\int_{\bD} \frac{(1-|v_\e|^2)^2}{2\e^2} + \frac{1}{2}\int_{\partial \bD} \left| \frac{\partial v_\e}{\partial \nu}\right|^2 \leq C(N,n,k,m,\omega_0). 
\end{align}
Consequently, following the analysis of Chapter X, \cite{BBH} using the estimate \eqref{eq:energy close}, Lemma \ref{lem:maximum princple}, and \eqref{eq:potential estimate} we conclude that upon possibly passing to a sub-sequence $\e_n \to 0,$ there exist distinct limiting vortex locations $(b_1,d_1), \cdots, (b_N,d_N)$ with $b_i \in \bD,$ and there exists $v_* \in C^\infty (\bD \backslash \{b_1, \cdots, b_N\}; \mathbb{S}^1)$ satisfying the following:
\begin{itemize}
\item For every $1 \leq p < 2,$ one has $\|v_{\e_n}\|_{W^{1,p}(\bD)} \leq C_p.$ 
\item For any compact set $K \subset \bD \backslash \{b_1, \cdots, b_N\}$ and for any natural number $\kappa \in \bN,$
\begin{align} \label{eq:ckcvgs}
\|v_{\e_n} - v_*\|_{C^\kappa(K)} \leq C(\kappa,K) \e_n^2,
\end{align}
for $\e_n$ sufficiently small.
\item $|v_*| = 1.$ This comes from \eqref{eq:potential estimate}. 
\item With $K$ and $\kappa$ as in the preceding item, 
\begin{align} \label{eq:gradterm estimate}
\left\|\frac{1- |v_{\e_n}|^2}{\e_n^2} - |\nabla v_*|^2 \right\|_{C^\kappa(K)} \leq C(\kappa, K) \e_n^2  \hspace{1cm} \mbox{ as } \e_n \to 0.
\end{align}
\item The limiting vortex locations $(b_1, \cdots, b_N) \in \call_{\wma}^\delta.$  
\item In fact, $v_{\e_n} \to v_*$ in $H^1(V)$ for any $V \subset \subset \bD \backslash \{b_1, \cdots, b_N\}.$ Consequently, using a standard diagonalization argument and using \eqref{eq:ckcvgs}-\eqref{eq:gradterm estimate},
\begin{align} \label{eq:limiting eq}
\Delta v_* + |\nabla v_*|^2 v_* &= 0, \hspace{1cm} \mbox{ in }\bD \backslash \{b_1,\cdots, b_N\}\\ \notag
v_* |_{\partial \bD} &= g_n.
\end{align}
Here is a different way to derive \eqref{eq:limiting eq}. Passing to the limit in $v_\e,$ we find that on $\bD \backslash \{b_1, \cdots, b_N\},$ we have 
\begin{align*}
\Delta v_* + |\nabla v_*|^2 v_* &= \omega_0 \left( k v_* + \frac{1}{m} (y^\perp \cdot \nabla v_*^\perp) \right). 
\end{align*}
Then, using the condition that $|v_*| = 1 $ on $\bD \backslash \{b_1,\cdots, b_N\},$ we compute that 
\begin{align*}
v_* \times \omega_0 \left( k v_* + \frac{1}{m} (y^\perp \cdot \nabla v_*^\perp) \right) &= \frac{\omega_0}{m} (v_*^{(1)},v_*^{(2)}) \times (y_1 \partial_2 - y_2 \partial_1)(-v_*^{(2)}, v_*^{(1)})\\ \notag
&= \frac{\omega_0}{m} (y_1 \partial_2 - y_2 \partial_1 )|v_*|^2\\ 
&= 0, 
\end{align*}
and so, 
$v_* \times \Delta v_* = 0. $ But the conditions $|v_*| = 1$ and $v_* \times \Delta v_* = 0$ easily imply that $v_*$ solves the harmonic map PDE, and consequently on $\bD \backslash \{b_1, \cdots, b_N\}$, 
$k v_* + \frac{1}{m} (y^\perp \cdot \nabla v_*^\perp)  = 0.
$
\end{itemize}
It now follows by H\'{e}lein's regularity theorem (cf. \cite{helein}) that $v_*$ is the canonical harmonic map associated with the points $\{(b_1,d_1),\cdots, (b_N,d_N)\}.$ We remark that this argument was also used in \cite{Struwe}, and earlier still, in \cite{kellerrubinsteinsternberg}. \par 
Finally, we show that $(b_1,d_1),\cdots, (b_N,d_N)$ is a critical point of $\calh^{n,\omega_0}.$ This will complete the proof of the theorem, since all critical points of $\calh^{n,\omega_0}$ in $\overline{\call}_{\wma}^\delta,$ lie on the critical circle $\call_{\wma}.$ This claim, will follow by an analogue of what is called the \textit{vanishing gradient property} in \cite{BBH}: we derive it using some Pohazaev-type identities, cf. \cite{BBH}, pp. 74. To this end, we know from the foregoing paragraph that near each $b_j,$ we may write $v_*(x) = e^{i(d_j \Theta_j + H_j)(x)},$ where $H_j$ is a real Harmonic function near (and including) $b_j,$ and $\Theta_j$ is the angle relative to polar coordinates with pole at $b_j.$ Moreover, by Theorem VIII.3 of \cite{BBH}, for \textit{any} $c \in (\bD^*)^M,$ where $M$ is some positive integer,
\begin{align}
\label{eq:derivative of renormalized energy}
DW(c) = 2\pi \left[ d_1 \left( - \frac{\partial H_1}{\partial x_2}(c_1), \frac{\partial H_1}{\partial x_1 }(c_1)\right),\cdots, d_M \left( - \frac{\partial H_M}{\partial x_2}(c_M), \frac{\partial H_M}{\partial x_1 }(c_M)\right)  \right].
\end{align}
In the present case, let $R > 0$ be sufficiently small so that $\overline{B_{2R}(b_1)}$ does not contain any other $b_j, j > 1,$ nor does it meet the boundary. We begin by taking the dot product of both sides of our elliptic PDE
\begin{align}
\label{eq:veps PDE}
\Delta v_\e + \frac{v_\e(1-|v_\e|^2)}{\e^2} = \omega_0 \left( k v_\e + \frac{1}{m} (y^\perp \cdot \nabla) v_\e^\perp\right)
\end{align}
with $\frac{\partial v_\e}{\partial y_1}.$  Integrating the result on $B_R(b_1),$ we find upon performing some integrations by parts, 
 \begin{align*}
 \int_{\partial B_R(b_1)} \frac{\partial v_\e}{\partial \nu} \cdot \frac{\partial v_\e}{\partial y_1} &- \frac{1}{2} \left| \nabla v_\e\right|^2(\nu \cdot e_1)- \int_{\partial B_R(b_1)} \frac{(1-|v_\e|^2)^2}{4\e^2} (\nu \cdot e_1) \\ &= \omega_0 \int_{\partial B_R(b_1)} \frac{k}{2} |v_\e|^2 (\nu \cdot e_1) - \frac{\omega_0}{m} \int_{B_R(b_1)} y_1 (Jv_\e). 
 \end{align*}
Passing to the limit $\e \to 0$ with $R$ fixed, we find using the Jacobian estimate, \cite[Theorem 2.1]{JerrardSoner}  for the last term that
\begin{align*}
\int_{\partial B_R(b_1)} \frac{\partial v_*}{\partial \nu} \cdot \frac{\partial v_*}{\partial y_1} - \frac{1}{2} |\nabla v_*|^2 (\nu \cdot e_1) &= \omega_0 \frac{k}{2}\int_{\partial B_R(b_1)} (\nu \cdot e_1) - \pi\frac{\omega_0}{m} d_1 (b_1)^{(1)} \\
&= -\pi\frac{\omega_0}{m} d_1 (b_1)^{(1)}
\end{align*}
In the above, $\nu$ denotes the unit normal vector to $\partial B_R(b_1).$ Let $\tau$ denote the unit tangent vector to $\partial B_R(b_1),$ chosen such that the pair $(\nu,\tau)$ preserves the orientation in $\bR^2.$ Then, arguing as in \cite{BBH} we find,
\begin{align*}
-\int_{\partial B_R(b_1)} \left( \frac{d_1}{R} \frac{\partial H_1}{\partial \tau} + \frac{1}{2} |\nabla H_1|^2 \right) (\nu \cdot e_1) + \int_{\partial B_R} \frac{\partial H_1}{\partial \nu} \cdot \left( d_1 \frac{\tau \cdot e_1}{R} + \frac{\partial H_1 }{\partial y_1} \right) = - \pi\frac{\omega_0}{m} d_1 (b_1)^{(1)}.
\end{align*}
On the other hand, since $\Delta H_1 = 0,$ multiplying by $\frac{\partial H_1}{\partial y_1}$ and integrating, we obtain 
\begin{align*}
- \int_{\partial B_R(b_1)} \frac{\partial H_1}{\partial \nu} \cdot \frac{\partial H_1}{\partial y_1}  + \frac{1}{2}\int_{\partial B_R(b_1)} |\nabla H_1|^2 (\nu \cdot e_1) = 0. 
\end{align*}
Adding these last two equations and dividing through by $d_1,$ we find 
\begin{align*}
-\frac{1}{R}\int_{\partial B_R} \frac{\partial H_1}{\partial \tau} (\nu \cdot e_1) + \frac{1}{R}\int_{\partial B_R} (\tau \cdot e_1) \frac{\partial H_1}{\partial \nu} = -\pi\frac{\omega_0}{m} (b_1)^{(1)}.
\end{align*}
By symmetry, a similar identity holds in the $e_2$ direction. Together, we find 
\begin{align}
- \frac{1}{R}\int_{\partial B_R(b_1)} \frac{\partial H_1}{\partial \tau} \nu + \frac{1}{R} \int_{\partial B_R(b_1)} \frac{\partial H_1}{\partial \nu} \tau = -\pi\frac{\omega_0 }{m} ((b_1)^{(1)}, (b_1)^{(2)}). 
\end{align}
Applying the mean value property to the harmonic functions $\frac{\partial H_1}{\partial \tau}, \frac{\partial H_1}{\partial \nu},$ we conclude that 
\begin{align}
2\pi(\nabla H_1(b_1))^{\perp} = -\pi\frac{\omega_0}{m} b_1. 
\end{align}
A similar identity holds at all $b_j, j = 1, \cdots, N.$ Putting these together with \eqref{eq:derivative of renormalized energy}, we find 
\begin{align}
D_{b_j}W(b_j) = 2\pi d_j (\nabla H_j(b_j))^\perp = -d_j\pi \frac{\omega_0}{m}b_j.
\end{align}
Dividing through by $\pi,$ we note from \eqref{eq:calh} that this says that $\nabla \calh^{n,\omega_0}(\mathbf{b},d) = 0.$  This completes the proof of the theorem. 
\end{proof}
\begin{remark}
\label{rem:modification} Assumption A (cf. Page 15) can be relaxed substantially: we were motivated about this relaxation by the remarks in \cite{LinAIHP}, pg. 618. Let $\wma \in (\bD^*)^{N_R}$ be a relative equilibrium. We may relax the Assumption A to the following: we assume that the real $2N_R \times 2N_R$ matrix $D^2 \calh^{n,\omega_0}(\wma)$ has at least one non-zero eigenvalue. Suppose that the number of $0$ eigenvalues of $D^2 \calh^{n,\omega_0}(\wma)$ is $Z,$ with $1 \leq Z < 2N_R.$ Next, consider the connected component of the level set of the functional $\calh^{n,\omega_0}$ containing the point $\wma$, i.e. 
\begin{align*}
\mathcal{M}_{\wma} := \{ \wmb \in (\bD^*)^{N_R}: \calh^{n,\omega_0}(\wmb) = \calh^{n,\omega_0}(\wma), \nabla \calh^{n,\omega_0}(\wmb) = 0\}. 
\end{align*}
Then the set $\mathcal{M}_{\wma}$ is clearly closed and bounded. Consequently, it is compact. Moreover, since the function $\calh^{n,\omega_0}$ is smooth in $(\bD^*)^{N_R},$ by the Rank Theorem, cf. \cite{Narasimhan} Theorem 1.3.14 or \cite{Boothby} Theorem 7.1, we conclude that $\mathcal{M}_{\wma}$ is a compact connected $Z-$dimensional smooth sub-manifold. We next argue that the rotational invariance of $\calh^{n,\omega_0}$ leads to $\mathcal{M}_{\wma}$ being foliated by simple smooth closed curves. 
\par Indeed, using the rotational invariance of the functional $\calh^{n,\omega_0},$ for any point $\wmp \in \mathcal{M}_{\wma},$ we find $e^{i\theta}\wmp \in \mathcal{M}_{\wma}.$ Consequently, as before, denoting $\call_{\wmp} := \{e^{i\theta} \wmp: \theta \in [0,2\pi) \},$ for any other point $\wmq \in \mathcal{M}_{\wma},$ if for some $\phi \in [0,2\pi)$ we have $e^{i\phi} \wmq \in \call_{\wmp},$ then necessarily $\wmq \in \call_{\wmp}.$ Hence, $\mathcal{M}_{\wma}$ is foliated by simple smooth curves of the form $\call_{\wmp},$ with $\wmp \in \mathcal{M}_{\wma}.$  
\par Given this set up, by the Tubular Neighborhood Theorem from Differential Topology  (cf. \cite{guilleminpollack}, Page 76), there exists a tubular neighborhood $\mathcal{M}_{\wma}^\delta$ with $\delta > 0$ small, which affords ortho-normal local coordinates similar to the one described for the set $\call_{\wma}^\delta$ in Step 1 of the proof of Theorem \ref{thm:existence}. Using these coordinates, we can once again write down a Taylor development of the function $\calh^{n,\omega_0}$ in the neighborhood $\mathcal{M}_{\wma}^\delta$ similar to \eqref{eq:hessian}. 
\par Now given a point $\wmp \in \mathcal{M}_{\wma},$ let $B_{\wmp} $ denote the set of points in $\mathcal{M}_{\wma}^\delta$ whose nearest point in $\mathcal{M}_{\wma}$ is $\wmp.$ By the Tubular Neighborhood Theorem, this set is well defined and lies in a plane normal to $\mathcal{M}_{\wma}$ at $\wmp,$ i.e. the orthogonal compliment in $\bR^{2N_R}$ of the tangent space $T_{\wmp} \mathcal{M}_{\wma}$ to $\mathcal{M}_{\wma}$ at $\wmp.$ We may use the eigenvectors of $D^2 \calh^{n,\omega_0}(\wmp)$ that are orthogonal to $T_{\wmp} \mathcal{M}_{\wma}$ to parametrize $B_{\wmp}.$ Then as in the proof of Theorem \ref{thm:existence}, we obtain the linking structure, using the positive and negative eigen-directions. It is then clear that the proofs of the two Theorems modify easily through obvious changes.  
\end{remark}
\begin{remark}[Quantifying convergence results] By way of comparing our results with those of \cite{jerrardspirn} and placing it in context, note that the periodic solutions to Gross-Pitaevskii that we have constructed follow point vortex dynamics for all time and for sufficiently small $\e > 0;$ however the statements of our results do not make this quantitative. The results of \cite{jerrardspirn} on the other hand are quantitative about the closeness of $Ju_\e(\cdot, t)$ to a sum of dirac masses, but only valid on time intervals on the order of $\ln \frac{1}{\e}.$ Combining our results with those of \cite{jerrardspirn}, one could obtain quantitative information for fixed positive but small $\e,$ valid for all time. 
\end{remark}
\begin{remark}[Applications of Theorems \ref{thm:existence} and \ref{thm:asymptotics}]
The relative equilibria to (PVF) that we have constructed in Section \ref{sec:PVF} imply that corresponding to each positive integer $k,$ our Theorems \ref{thm:existence}-\ref{thm:asymptotics} establish the existence of time-periodic solutions to Gross Pitaevskii equation \eqref{eq:GP}-\eqref{eq:bc} with $k-$fold symmetry, that follow relative equilibria to (PVF) having two rings, one containing $k$ vortices of degree $+1$ and another containing $k$ vortices of degree $-1.$ 
\end{remark}
\begin{appendices}
\section{Some Technical Lemmas}
In this appendix, we collect various technical results we need throughout the course of the paper. The first lemma is a maximum principle for solutions to \eqref{ansatzpde}. 
\begin{lemma} \label{lem:maximum princple}
Let $\omega \leq 0 $ and $k \geq 0$. Let $v_\e$ be a smooth solution to \eqref{ansatzpde}. Then for $\e > 0$ sufficiently small,
\begin{align} \label{eq:max principle}
|v_\e|^2 \leq 1 + \e^2 |\omega| k, 
\end{align}
and there exists a constant $C >0 $ independent of $\e >0,$ 
\begin{align} \label{eq:grad estimate}
|\nabla v_\e| \leq \frac{C}{\e}.
\end{align} 
\end{lemma} 
\begin{proof}
We first prove \eqref{eq:max principle}, whose proof follows from a  partial differential identity. For brevity we suppress dependence of $v$ on $\e.$ Set $u = |v|^2.$ Since $ \Delta u = 2 |\nabla v|^2 + 2 v \cdot \Delta v$, we calculate that 
\begin{align} 
\Delta u &= 2 \omega k u - 2 \frac{u(1-|v|^2)}{\e^2}  - \frac{\omega^2}{m^2}(y_1^2 + y_2^2)u + P,
\end{align}
where $P$ is positive:
\begin{align*}
P &= |\nabla v|^2 +\left( -\frac{\omega}{m} v_1 y_1 + \partial_{y_2} v_2\right)^2 \\ &+ \left( \frac{\omega}{m} y_2v_1 + \partial_{y_1} v_2\right)^2 + \left( -\frac{\omega}{m} v_2 y_1 + \partial_{y_2} v_1\right)^2 + \left( \frac{\omega}{m} v_2y_2 + \partial_{y_1} v_1\right)^2. 
\end{align*}
It follows that 
\begin{align} \label{eq:maxprincistep}
\Delta u(y) - u(y) \left(2\omega k  - 2\frac{(1-|v(y)|^2)}{\e^2} - \frac{\omega^2(y_1^2 + y_2^2)}{m^2} \right) \geq 0.
\end{align}
Note $u|_{\bdry} \equiv 1. $ Assume to the contrary that \eqref{eq:max principle} fails, so that there exists a point $\hat{y} \in \bD$ such that 
\begin{align} \label{eq:contradictionhyp}
|v_\e(\hat{y})|^2 = \max_{\overline{\bD}} |v_\e|^2 > 1 - \e^2 \omega k;
\end{align}
recall that we are assuming $\omega \leq 0.$ If $\hat{y} = (0,0),$ then, $\Delta u(0,0) \leq 0$ and we reach a contradiction to the inequality \eqref{eq:maxprincistep}. Suppose now that $\hat{y} \neq (0,0).$ In order to reach a contradiction, we will invoke a conformal map and we make the usual identification $y = (y_1, y_2) \in \bR^2 \sim \bC$ with the complex number $y = y_1 + i y_2.$ 
With this notational set up, we continue with \eqref{eq:contradictionhyp}, for $\hat{y} \neq (0,0).$ Since $u \equiv 1$ on $\partial \bD,$ we may assume $\hat{y} \in \bD.$ Let $\Lambda : \bD \to \bD$ be a conformal automorphism of the disc, i.e. a M\"{o}ebius transformation such that $\Lambda(\hat{y}) = (0,0).$ Consider the holomorphic change of variables $z := \Lambda(y).$ Then, define the function $U: \bD \to \bR$ using the formula
\begin{align*}
U(z) := u(y).
\end{align*} 
Recalling that $\Lambda$ is holomorphic, an easy calculation yields that $\Delta_y u(y) = |\nabla_y \Lambda (y)|^2 \Delta_z U(z).$ Since the conformality of $\Lambda$ prevents it from having vanishing derivative, we can re-write \eqref{eq:maxprincistep} at $z = \Lambda(y)$ as 
\begin{align*}
\frac{1}{|\nabla \Lambda(y)|^2 } \Delta_z U(z) - U(z) \left(2 \omega k - 2\frac{1 - U(z)}{\e^2} - \frac{\omega^2}{m^2}|z|^2 \right) \geq 0. 
\end{align*}
Naturally, $U|_{\partial \bD} = 1.$ Since $u$ has a maximum at $\hat{y},$ the function $U$ has a maximum at $\hat{z} = \Lambda(\hat{y}) = (0,0),$ reducing it to the earlier studied case. This then completes the proof of \eqref{eq:max principle} by the same argument as before. Having completed the proof of \eqref{eq:max principle}, the proof of \eqref{eq:grad estimate} now follows from elliptic estimates (see for instance \cite[Chapter 8]{GilbargTrudinger}). 
\end{proof}
\begin{remark}
The referee has pointed out to us that the preceding maximum principle also holds when $\omega k \geq 0,$ and also, $2 \omega k \geq \frac{\omega^2}{k^2}.$ In this case, one has that $|v_\e|\leq 1.$ 
\end{remark}
Our next Proposition is a construction which for a given collection of points $p = (p_1, \cdots, p_N)$ and associated degrees $d = (d_1,\cdots, d_N)$, constructs a function $w^\e_{p}: \bD \to \bD$ with $\scrj(w^\e_p)$ very nearly equal to $\scrj_0(p,d)$ and also has very close to minimal Ginzburg Landau energy $E_\e.$ 
\begin{proposition}
\label{prop:BBH construction} 
Let $N \geq 1$ be an integer and let $p = (p_1,\cdots, p_N) \in (\bD^*)^N$ and $d = (d_1,\cdots, d_N) \in \{\pm 1\}^N,$ with $\sum_{i=1}^N d_i = n.$  Then for $\e > 0 $ small, there exists a function $w^\e_p \in H^1_{g_n}(\bD, \bC) $ which satisfies the following. 
\begin{enumerate}
\item $w^\e_p(x) = 0$ if and only if $x = p_i.$ Outside the disjoint balls $(B_{2\e}(p_j))_{j=1}^N,$ one has that $|w^\e_p| = 1,$ and consequently $\deg(w^\e_p,p_i)$ is well defined.  We have $\deg(w^\e_p,p_i) = d_i.$ 
\item Let $\scrj_0(p,d)$ be as in \eqref{eq:fdmomentum}. Then 
\begin{align} \label{momentumoptimal}
\left|\scrj(w^\e_p) - \frac{\pi}{m}\scrj_0(p,d)\right| = o_\e(1)
\end{align}
as $\e \to 0.$ 
\item The Ginzburg Landau energy $E_\e(w^\e_p)$ satisfies the asymptotic development
\begin{align} \label{energyoptimal}
E_\e(w^\e_p) = N \left(\pi \log \frac{1}{\e} + \gamma\right) + W(p,d) + o_\e(1)
\end{align}
as $\e \to 0$ where $\gamma$ is a universal constant identified in \cite{BBH}. 
\item (Symmetry property) Let now $k$ be a divisor of $n$ and let $m = \frac{n}{k}.$ Suppose that the points $p$ and their associated degrees $d$ are preserved under rotation by angles $\frac{2\pi j}{k}, 0 \leq j \leq k-1.$ Then we may arrange that 
\begin{align}
\label{sym1}
w^\e_p(r,\frac{2\pi}{k} + \theta) &= w^\e_p(r,\theta), \hspace{1cm} r \in [0,1], \theta \in \bR.
\end{align}
\item Suppose that in addition to the assumptions in (4), we know that the points $p_i$ are arranged in concentric regular $k-$gons that are aligned in the sense that vortices in successive rings are arranged along the same rays emanating from the origin (cf. Figure 1). Then we may also arrange that 
\begin{align}
\label{sym2}
w^\e_p(r,\frac{\pi j}{k} + \theta) &= \overline{w^\e_p\left(r,\frac{\pi j}{k} - \theta\right)}, \hspace{.6cm} \mbox{ for odd } j,\mbox{ and } 0 \leq \theta \leq \frac{\pi}{k}.
\end{align}
%\end{align}
\end{enumerate}
\end{proposition}
Since the proof of this Proposition follows the proof of Lemma 5.1 and Proposition 4.4 in \cite{GS} rather closely, we only make some comments regarding its proof. The starting point is the function $\Phi_0$ defined in \eqref{Phi0}-\eqref{eq:Phi0norm}. Using the Poisson integral formula corresponding to Neumann boundary conditions, we find 
\begin{align*}
\Phi_0(z) = \sum_{i=1}^N d_i \log|(b_i - z)(1 - \overline{b}_i z)|.
\end{align*}
Following Lemma 5.1 of \cite{GS}, one defines the phase function $\tilde{\chi}.$ The function $w^\e_p,$ is then defined to be $e^{i\tilde{\chi}}$ away from the vortices, and $\frac{\rho_j}{\e}e^{id_j \Theta_j}$ on vortex balls of radius $O(\e)$ centered $p_j,$ with $(\rho_j,\Theta_j)$ denoting polar coordinates at $p_j.$ The symmetry properties \eqref{sym1}-\eqref{sym2} arise from symmetry properties of the functions $\Phi_0$ and $\tilde{\chi}.$ The proof of the energy estimate \eqref{energyoptimal} proceeds exactly as in \cite{GS}. Here is a sketch of the proof of the momentum estimate \eqref{momentumoptimal}: Given points $p_1,\cdots, p_N$ as in the statement of Proposition \ref{prop:BBH construction}, consider concentric "rings" $R_i^\e$ centered at the origin, of inner and outer radii $|p_i| - 2\e$ and $|p_i| + 2\e$ respectively. Using the $\log$ bound on energy \eqref{energyoptimal}, and the Cauchy Schwarz inequality, it follows that the contribution of these rings to $\scrj(w^\e_p) = o_\e(1)$ as $\e \to 0.$ On the other hand, in between rings $|w^\e_p| \equiv 1,$ while $w_\e^p \cdot (y^\perp \cdot \nabla (w_\e^p)^\perp) = \frac{\partial \tilde{\chi}}{\partial \theta},$ where $\theta$ denotes polar angle relative to the usual polar coordinates. The desired estimate follows by noting that if $\mathcal{C}$ is any circle centered at the origin, and contained in $\bD \backslash \cup_{i=1}^N R^\e_i$ then $\int_\mathcal{C} \frac{\partial \tilde{\chi}}{\partial \theta} = 2 \pi d_C$ where $d_C$ is the total degree of vortices contained inside the circle $C.$ One then computes using $\sum d_i = n = km$ that
\begin{align*}
\scrj(w^\e_p) &= -\frac{1}{2} \big\{ k \pi - \frac{\pi}{m}\big(  d_1 (|p_2|^2 - |p_1|^2) + (d_1 + d_2)(|p_3|^2 - |p_2|^2) + \cdots \\  &+ (d_1 + \cdots + d_{N-1})(|p_N|^2 - |p_{N-1}|^2) + (d_1 + \cdots + d_N) (1 - |p_N|^2)\big) \big\} +o_\e(1)\\
\Rightarrow \scrj(w^\e_p)&=  -\frac{\pi}{2} \left\{ k - \frac{n}{m} + \frac{
1}{m}\big(d_1 |p_1|^2 + \cdots + d_N |p_N|^2 \big)  \right\} + o_\e(1)\\
&=  \frac{\pi}{m}\scrj_0(p,d) + o_\e(1).
\end{align*}
\par 
Some remarks concerning this proposition are in order. First, fixing $N, d_1, \cdots, d_N$ the map $p \mapsto w^\e_p: (\bD^*)^{\bN} \to H^1_{g_n}(\bD,\bC)$ is continuous, where the domain is equipped with the standard Euclidean topology, and the range is equipped with the Sobolev norm topology. Second, we will use this construction in both our approaches: in the first as an admissible competitor, while in the second as initial datum to a heat flow that we will consider there. 
\par
We next state a modification of the Projection Lemma in  \cite{LinMinMax}. It serves to construct an approximate inverse to the map $p \mapsto w^\e_p$ obtained in the preceding proposition; given a function in $u \in H^1_g,$ under certain assumptions on the Ginzburg-Landau energy of $u,$ the projection lemma identifies its essential zeroes, which serve as vortices of the function $u$ upto small errors in $\e.$ 
\par 
Suppose $n,N$ are as in Proposition \ref{prop:BBH construction}. Naturally, we have $n \leq N,$ but typically $N$ can be larger than $n$ and the Projection Lemma as stated in \cite{LinMinMax} cannot be expected to hold as is. Indeed, given $u \in H^1_{g_n}(\bD, \bC),$ we have the easy estimate 
\begin{align*}
E_\e(u) \geq \min_{\mathfrak{u} \in H^1_{g_n}} E_\e(\mathfrak{u}) \geq n \pi \log \frac{1}{\e}  - C,
\end{align*}
where the last inequality follows from \cite{BBH} and $C$ is independent of $\e.$ Suppose in addition that $u$ satisfies the estimate 
\begin{align*}
E_\e(u) \leq N \pi \log \frac{1}{\e} + K. 
\end{align*}
Clearly, there exists exactly one integer $m \in [n,N]$ such that $u$ belongs the energy band
\begin{align*}
u \in \mathcal{S}_m := \left\{ u \in H^1_{g_n}(\bD,\bC): \left(m - \frac{1}{2}\right) \leq \frac{E_\e(u)}{\pi \log \frac{1}{\e}} < \left( m + \frac{1}{2} \right)  \right\}. 
\end{align*}
The next Lemma deals with functions that belong to the class $\mathcal{S}_N$ that in addition satisfy the following assumptions:
\begin{enumerate}[label=(\subscript{A}{\arabic*})]
\item $|u(x)| \leq 2$ on $\bD$ and
\item there exists $c_0 > 0$ such that for any $x_0 \in \{x \in \bD, |u(x)| \leq \frac{1}{2} \} $ the set $B_{c_0 \e}(x_0) \subset \{|u(x)| \leq \frac{3}{4} \}. $ 
\end{enumerate}
\begin{proposition}[Lemma 1.3 \cite{LinMinMax}] \label{lem: proj lemma}
Let $D$ be a compact topological submanifold with boundary in $\bR^{2N},$ and let $h:D \to H^1_{g_n}(\bD,\bC)$ be a continuous map such that $h(D) \subset S_N,$ and in addition, elements of $h(D)$ satisfy assumptions $(A_1)$ and $(A_2)$ above. Then there exist positive numbers $\lambda = \lambda(c_0,K,C, n,N), \e_0 $ and $\alpha_0$ and a continuous map $\Pi:h(D) \to W_\lambda := \{ b \in (\bD^*)^N : W(b,d) \leq \lambda\}$ when $\e \in (0,\e_0).$ Furthermore, for any $p \in D,\Pi(h(p))$ is at most $4\e^{\alpha_0}$ away from the essential zeroes of $h(p).$ 
\end{proposition}
We omit the proof of this Lemma since it follows Lemma 1.3 of \cite{LinMinMax} rather closely. 
\end{appendices}

\end{document}